\documentclass{amsart}
\usepackage{dsfont}
\usepackage{graphicx}
\usepackage{caption}
\newcommand{\e}{\varepsilon}
\newcommand{\dist}{\text{dist}}
\newcommand\R{\mathbb{R}}
\newcommand\w{\omega}
\newcommand\g{\gamma}

\theoremstyle{plain}
\newtheorem{thrm}{Theorem}[section]
\newtheorem{lmm}[thrm]{Lemma}
\newtheorem{crllr}[thrm]{Corollary}
\newtheorem{prpstn}[thrm]{Proposition}

\theoremstyle{definition}
\newtheorem{dfntn}[thrm]{Definition}

\newtheorem{xmpl}[thrm]{Example}

\newtheorem{rmrk}[thrm]{Remark}

\begin{document}
\title{Motion of discrete interfaces in low-contrast random environments}
%
\author{Matthias Ruf}\address{Zentrum Mathematik - M7, Technische Universit\"at M\"unchen, Boltzmannstrasse 3, 85748 Garching, Germany}
\email{mruf@ma.tum.de}
\begin{abstract}  
We study the asymptotic behavior of a discrete-in-time minimizing movement scheme for square lattice interfaces when both the lattice spacing and the time step vanish. The motion is assumed to be driven by minimization of a weighted random perimeter functional with an additional deterministic dissipation term. We consider rectangular initial sets and lower order random perturbations of the perimeter functional. In case of stationary, $\alpha$-mixing perturbations we prove a stochastic homogenization result for the interface velocity. We also provide an example which indicates that stationary, ergodic perturbations do not yield a spatially homogenized limit velocity for this minimizing movement scheme.
\end{abstract}
\maketitle
\section*{Introduction}
In 1993 Almgren, Taylor and Wang introduced a notion of minimizing movements suitable to describe geometric motions of interfaces driven by curvature effects (see \cite{ATW93}). In a nutshell it can be summarized as follows: Given a fixed time step $\tau>0$ and an initial set $A_0\subset\R^d$, one constructs recursively a sequence of sets $\{A_k^{\tau}\}_k$ minimizing an energy functional of the form
\begin{equation}\label{introminmov}
A\mapsto E^{\tau}(A,A^{\tau}_{k-1})=\int_{\mathcal{F} A}\varphi(\nu(x))\,\mathrm{d}\mathcal{H}^{d-1}+\frac{1}{2\tau}\int_{A\Delta A^{\tau}_{k-1}}\dist(x,\partial A^{\tau}_{k-1})\,\mathrm{d}x,
\end{equation} 
where $\nu(x)$ is the normal vector at the point $x$ in the reduced boundary $\mathcal{F} A$ (we refer to \cite{AFP} for a precise definition) and $\varphi$ is a suitable surface density. The basic idea behind this approach is the following: While minimizing the surface functional shrinks the set, the bulk term forces the boundary of the minimizer to be close to the boundary of the previous set. Passing to the limit as $\tau\to 0$ for the piecewise constant interpolations one obtains a time dependent family $A(t)$ of sets that evolves by a weighted curvature (depending on $\varphi$), provided the initial set $A_0$ is regular enough and $\varphi$ is elliptic and smooth. In the isotropic case one obtains the well-known motion by mean curvature. This minimizing movement procedure was later on exported to random environments by Yip in \cite{yip} as follows: at each discrete time step, a minimizer of the energy in (\ref{introminmov}) is computed and then this set is perturbed by a random diffeomorphism.

In the recent paper \cite{BGN} Braides, Gelli and Novaga applied the above minimizing movement scheme within a deterministic, discrete environment. In this setting the environment is the scaled two-dimensional lattice $\e\mathbb{Z}^2$.  The surface term in (\ref{introminmov}) is replaced by a discrete interfacial energy which, in its simplest form, is derived from the classical nearest neighbor Ising model for spin systems and can be written formally as
\begin{equation}\label{introinteraction}
P_{\e}(u)=\frac{1}{4}\sum_{\substack{\e i,\e j\in\e\mathbb{Z}^2\\|i-j|=1}}\e|u(\e i)-u(\e j)|,
\end{equation}
where $u:\e\mathbb{Z}^2\to\{\pm 1\}$ is the spin variable. Note that the energy in (\ref{introinteraction}) takes into account only nearest neighbor interactions. Therefore it coincides with the perimeter of the set $\{u=+1\}$ and the relationship to the continuum model is given by identifying the spin variable with this level set. The distance-function in the bulk term in (\ref{introminmov}) is replaced by a discrete version of the $l^{\infty}$-distance to the boundary precisely defined in (\ref{distance}). From a physical point of view this setup can be seen as a simplified model to describe the motion of boundaries of the level sets of the spin variables, which represent the magnetic domain walls at the discrete level. Since the discrete perimeter inherits the anisotropy of the lattice, this minimizing movement scheme is related to crystalline motions, where $\varphi$ is not smooth (see \cite{AT95,BCCN,CMP} in the continuum case). Note that the continuum limit (or $\Gamma$-limit; see \cite{GCB}) of the energies in (\ref{introinteraction}) is given by the crystalline perimeter, that is 
\begin{equation}\label{crystalline}
P(u)=\int_{S_u}|\nu(x)|_1\,\mathrm{d}\mathcal{H}^1=\int_{\mathcal{F} \{u=1\}}|\nu(x)|_1\,\mathrm{d}\mathcal{H}^{1},
\end{equation}
where $|\nu|_1$ denotes the $l^1$-norm of $\nu$ (see \cite{ABC}). In \cite{BGN} the authors observed that the asymptotic behavior of the discrete flows depends heavily on the scaling between $\e,\tau$ when $\e,\tau\to 0$ simultaneously. When $\e/\tau\to 0$ fast enough the motion is governed by the $\Gamma$-limit, that means one obtains the continuum motion by crystalline curvature. If, on the other hand, $\e/\tau\to +\infty$ fast enough, the motions are pinned by the presence of many local minimizers in the discrete environment. This phenomenon is similar to any gradient flow that starts in a local minimum. We remark that in general such a priori results are abstract and the necessary speed of convergence/divergence might be unknown. We refer the reader to Chapter 8 in \cite{Locmin} for a parade of further examples on this issue. However the exact behavior was found in \cite{BGN}. The critical scaling for the discrete perimeter energies (\ref{introinteraction}) is $\e\sim\tau$, where pinning effects due do discreteness as well as a quantized crystalline motion can occur (see also Theorem \ref{bgn1} below).

In this paper we start studying the effect of a random discrete environment on the continuum limit flow. We take a different approach compared to \cite{yip} and associate the random effects directly to the lattice points. In \cite{ACR} Alicandro, Cicalese and the author performed a discrete-to-continuum analysis for a large class of ferromagnetic Ising-type energies including (\ref{introinteraction}) where the interacting particles are located at the points of a so-called stochastic lattice $\e\mathcal{L}(\w)$ instead of the periodic $\e\mathbb{Z}^2$ (or more generally $\e\mathbb{Z}^d$). In particular, assuming the stochastic lattice to be stationary with respect to translations, one can prove the existence of an homogeneous limit surface energy that turns out to be deterministic under some ergodicity assumptions. In that case the continuum limit takes the form
\begin{equation*}
P_{{\rm hom}}(u)=\int_{S_u}\varphi_{\rm hom}(\nu(x))\,\mathrm{d}\mathcal{H}^1.
\end{equation*}
Hence to continuum limit resembles the surface term in (\ref{introminmov}) even though in general $\varphi_{{\rm hom}}$ might be non-smooth. Our aim is to include dynamical effects in order to describe the curvature-driven motion of magnetic domain walls. The natural approach in the spirit of \cite{ACR,BCR} would be to replace the periodic lattice in the definition of $P_{\e}$ by a stationary random lattice $\e\mathcal{L}(\w)$ with suitable short-range interactions. This seems to be a very challenging problem. Thus we start with a much simpler model by adding very small random perturbations directly to the periodic lattice model, that means we will study the minimizing movement of a random discrete perimeter of the form
\begin{equation*}
P_{\e}^{\w}(u)=\frac{1}{4}\sum_{\substack{\e i,\e j\in\e\mathbb{Z}^2\\|i-j|=1}}\e(1+\e c_{ij}(\w))|u(\e i)-u(\e j)|.
\end{equation*}
For the precise assumptions on the random field $c_{ij}$ we refer to Section \ref{sec:prelim}. Note that in this scaling the random perturbations are are lower order term as they are scaled by $\e$. Nevertheless it turns out that they may influence at least the velocity of the limit motion. The reason why we don't let the bulk term be affected by the randomness as well, comes from the physical interpretation we give to this model as motion of aligned spins and differs from lattice particle models: While the interaction between particles may be affected by some random noise deriving from microscopic fluctuations, the energy to flip a spin should be constant, depending only on how many boundary layers are flipped in one time step. The interpretation of the bulk term in the energy in this setting is the following: Flipping the first layer of spins costs the least energy while the following layers are energetically more expensive. Of course this interpretation makes sense only if one can prove that in presence of randomness sets shrink by flipping spins close to the boundary by a certain number of layers. This is the case in the deterministic setting considered in \cite{BGN}.

For the sake of simplicity, we investigate the evolution when the initial set is a coordinate rectangle, that means a rectangle with all sides parallel to one of the coordinate axes. In Theorem \ref{motionthm1} we prove that under stationarity and quantified mixing assumptions as well as a suitable uniform bound on the random field $c_{ij}$, the limit motion law is deterministic and coincides with the quantized crystalline flow obtained in \cite{BGN}. This however depends strongly on the fact that the random field is stationary with respect to the translation group on $\mathbb{Z}^2$. In Section \ref{diffstatio} we show that the velocity changes if we restrict stationarity to a subgroup of the form $m\mathbb{Z}^2$ with $m\geq 2$. Anyhow, we stress that our results should be seen as a stability result of the deterministic problem rather than an exhaustive description of the possible effects of randomness on the limit flow. Indeed, randomness can influence the motion drastically. For example, in \cite{ACR} it is proven that when we replace the square lattice $\mathbb{Z}^2$ by a suitable isotropic stochastic lattice $\mathcal{L}(\w)$, then, up to a multiplicative constant, the discrete perimeters $\Gamma$-converge to the Euclidean instead of the crystalline perimeter. Thus, with an appropriate choice of discrete distance, one should not expect a crystalline motion anymore in the limit but rather some type of motion by mean curvature, at least if $\e<<\tau$ and the initial sets $A^{\e,\tau}_0$ converge to a smooth set.  
To highlight possible difficulties even in the present very weak random setting, we provide an example of stationary, ergodic perturbations that indicate strong non-uniqueness effects dropping the mixing hypothesis. In this case the functional describing the pointwise motion may not converge (Example \ref{nonconv}) so that the discrete velocity remains random, but still averaging over an increasing number of time steps one may obtain a homogenized limit velocity making further assumptions. Moreover, in Remark \ref{genericbound} we briefly discuss what might happen when we consider random fields $c_{ij}$ satisfying a generic $L^{\infty}$-bound.  

\section{Notation and preliminaries}\label{sec:prelim}
In this section we introduce our model and recall some definitions from probability theory as well as existing results in the deterministic setting.
\subsection{The random model}\label{sec:random-model}
First we set some notation. Given an interval $I$ and a function $f:I\to\mathbb{R}$ we set
\begin{equation*}
f^{(-)}(x)=\liminf_{y\to x}f(y),\quad f^{(+)}(x)=\limsup_{y\to x}f(y).
\end{equation*}
We set $Q_{\delta}(x)=x+[-\frac{\delta}{2},\frac{\delta}{2})^2$ as the half-open coordinate square centered at $x$ with side length $\delta$. For a real number $y\in\R$, we let $\lfloor y\rfloor$ be its integer part and $\lceil y\rceil:=\lfloor y\rfloor+1$. By $|\cdot|$ we denote the Euclidean norm on $\mathbb{R}^2$. If $B\subset\mathbb{R}^2$ is a Borel set we denote by $|B|$ its Lebesgue measure and by $\mathcal{H}^1(B)$ its $1$-dimensional Hausdorff measure.
Moreover, we set $\text{d}_{\mathcal{H}}(A,B)$ as the Hausdorff metric between two sets $A,B$. The symmetric difference of two sets $A,B$ is denoted by $A\Delta B$. We set $\mathds{1}_B$ as the characteristic function of $B$, and we denote by $\mathds{E}[X]$ the first moment of a random variable $X$. 

We now specify the framework for our model. Let $(\Omega,\mathcal{F},\mathbb{P})$ be a complete probability space. As pointed out in the introduction, we consider the easiest type of normalized ferromagnetic energies accounting only for nearest neighbor interactions. Given $\w\in\Omega$ and a function $u:\e\mathbb{Z}^2\to\{\pm 1\}$ we set
\begin{equation*}
P_{\e}^{\w}(u)=\frac{1}{4}\sum_{\substack{i,j\in\mathbb{Z}^2\\|i-j|=1}}\e\left(1+\e c_{ij}(\w)\right)|u(\e i)-u(\e j)|,
\end{equation*}
where the $c_{ij}:\Omega\rightarrow\R$ are uniformly bounded random variables satisfying a suitable $\alpha$-mixing assumption specified in (\ref{mixingdecay}). Note that without loss of generality we may assume that $c_{ij}=c_{ji}$ for all $|i-j|=1$. We define
\begin{equation*}
\mathcal{A}_{\e}:=\{A\subset\R^2:\;A=\bigcup_{i\in\mathcal{I}}Q_{\e}(i)\text{ for some }\mathcal{I}\subset\e\mathbb{Z}^2\}.
\end{equation*}
This class $\mathcal{A}_{\e}$ is closed under unions and intersections. Identifying a function $u:\e\mathbb{Z}^2\rightarrow\{\pm1\}$ with the set $A$ given by
\begin{equation*}
A:=\bigcup_{u(\e i)=+1} Q_{\e}(i)\in A_{\e},
\end{equation*}
we can interpret $P_{\e}^{\w}$ as a random perimeter defined on $\mathcal{A}_{\e}$ via $P_{\e}^{\w}(A):=P_{\e}^{\w}(u)$.
\\
If $\mathcal{Z}^2:=\{\xi=\frac{i+j}{2}:\;i,j\in\mathbb{Z}^2,\,|i-j|=1\}$ denotes the dual lattice of $\mathbb{Z}^2$, we can rewrite the random perimeter as a sum over points on the boundary $\partial A$ via
\begin{equation}\label{ondual}
P^{\w}_{\e}(A)=\sum_{\substack{\xi\in\mathcal{Z}^2\\\e\xi\in\partial A}}\e(1+\e c_{\xi}(\w)),
\end{equation}
where with a slight abuse of notation we set $c_{\xi}(\w):=c_{ij}(\w)$. From now on we assume the random variables to be indexed by the dual lattice. Given $A\in\mathcal{A}_{\e}$ it will be useful to define the properly scaled random perimeter also on portions of the boundary $\Gamma\subset\partial A$ setting
\begin{equation*}
p_{\e}^{\w}(\Gamma)=\sum_{\xi\in\mathcal{Z}^2:\;\e\xi\in\Gamma}\e c_{\xi}(\w).
\end{equation*}
With this notion, it holds that $P^{\w}_{\e}(A)=\mathcal{H}^1(\partial A)+\e p^{\w}_{\e}(\partial A)$.

In order to adapt the idea of Almgren, Taylor and Wang for studying curvature-driven motions, we have to define a suitable discrete distance between sets. As in \cite{BGN} we take a discrete version of the $l_{\infty}$-distance. To this end, first note that for every $x\in\R^2$ there exists a unique point $i\in\e\mathbb{Z}^2$ such that $x\in Q_{\e}(i)$. Given a set $A\subset\mathcal{A}_{\e}$ we define the value of the measurable function $d_{\infty}^{\e}(\cdot,\partial A):\R^2\rightarrow[0,+\infty)$ at $x\in Q_{\e}(i)$ by
\begin{equation} \label{distance}
d_{\infty}^{\e}(x,\partial A):=
\begin{cases}
\inf\{\|i-j\|_{\infty}:\;j\in\e\mathbb{Z}^2\backslash A\} &\mbox{if $i\in A$,}
\\
\inf\{\|i-j\|_{\infty}:\;j\in\e\mathbb{Z}^2\cap A\} &\mbox{if $i\notin A$.}
\end{cases}
\end{equation}
Observe that by definition $d_{\infty}^{\e}(x,\partial A)\in\e\mathbb{N}$.

Now we can define the total energy to be considered in the minimizing movement scheme. Given a mesh size $\e>0$, a time step $\tau>0$, sets $A, F\in\mathcal{A}_{\e}$ and $\w\in\Omega$ we set
\begin{equation*} 
E^{\w}_{\e,\tau}(A,F):=P^{\w}_{\e}(A)+\frac{1}{\tau}\int_{A\Delta F}d_{\infty}^{\e}(x,\partial F)\,\mathrm{d}x.
\end{equation*}
For a fixed (possibly random) initial set $A_{\e}^0(\w)$, we introduce the following discrete-in-time minimization scheme:
\begin{enumerate}
	\item[(i)] $A_{\e,\tau}^0(\w):=A_{\e}^0(\w),$
	\item[(ii)] $A_{\e,\tau}^{k+1}(\w)$ minimizes $A\mapsto E^{\w}_{\e,\tau}(A,A_{\e,\tau}^{k}(\w))$.
\end{enumerate}
Note that this procedure might not be unique. The discrete flat flow is defined as the piecewise constant interpolation
\begin{equation*}
A_{\e,\tau}(t)(\w):=A_{\e,\tau}^{\lfloor t/\tau\rfloor}(\w).
\end{equation*}
As a by-product of the analysis performed in \cite{BGN}, the most interesting regime is $\tau \sim \e$. Hence we assume for simplicity that 
\begin{equation*}
\tau=\g\,\e\quad\text{for some }\g>0
\end{equation*}
and omit the dependence on $\tau$ in the notation introduced above. For a complete analysis we have to require that the coefficient field satisfies the bound
\begin{equation}\label{pertubationsize}
\sup_{\xi}|c_{\xi}(\w)|<\frac{1}{4\gamma}\quad\mathbb{P}\text{-almost surely.}
\end{equation}
We remark that some of the results in this paper are valid for a generic $L^{\infty}$-bound but unfortunately these are not enough to characterize the motion.
\begin{rmrk}
	{\rm Using the boundedness of the random coefficients it is easy to see that $P^{\w}_{\e}(A)$ has the same $\Gamma$-limit in the $L^1$-topology as $P_{\e}$ defined in (\ref{introinteraction}), so that it converges to the crystalline perimeter (\ref{crystalline}).}	
\end{rmrk}

Now we introduce several stochastic properties of the random field $\{c_{\xi}\}_{\xi\in\mathcal{Z}^2}$. In general, given an indexed sequence $\{X_i\}_{i\in J}$ and $I\subset J$, we set $\mathcal{F}_{I}=\sigma\left(X_i:\;i\in I\right)$ as the $\sigma$-algebra generated by the random variables $\{X_i\}_{i\in I}$. We recall the following definitions from ergodic/probability theory:
\begin{dfntn}
	We say that a family of measurable functions $\{\tau_z\}_{z\in \mathbb{Z}^2},\tau_z:\Omega\to\Omega$, is an additive group action on $\Omega$ if
	\begin{equation*}
	\tau_{z_1+z_2}=\tau_{z_2}\circ\tau_{z_1}\quad\forall\, z_1,z_2\in\mathbb{Z}^2.
	\end{equation*} 
	Such an additive group action is called measure preserving if
	\begin{equation*}
	\mathbb{P}(\tau_z B)=\mathbb{P}(B)\quad \forall\, B\in\mathcal{F},\,z\in\mathbb{Z}^2.
	\end{equation*}
	Moreover $\{\tau_z\}_{z\in\mathbb{Z}^2}$ is called ergodic if, in addition, for all $B\in\mathcal{F}$ we have
	\begin{equation*}
	(\tau_z(B)=B\quad\forall\, z\in \mathbb{Z}^2)\quad\Rightarrow\quad\mathbb{P}(B)\in\{0,1\}.
	\end{equation*}
\end{dfntn}
We will need a quantitative version of Birkhoff's ergodic theorem that can be expressed through the notion of $\alpha$-mixing sequences.
\begin{dfntn}
	A sequence $\{X_j\}_{j\in\mathbb{N}}$ is said to be $\alpha$-mixing if there exists a sequence $\alpha(n)\to 0$ such that for all sets $I_1,I_2\subset\mathbb{N}$ with $\dist(I_1,I_2)\geq n$ it holds that
	\begin{equation*}
	\sup\{|\mathbb{P}(A\cap B)-\mathbb{P}(A)\mathbb{P}(B)|:\,A\in\mathcal{F}_{I_1},\,B\in\mathcal{F}_{I_2}\}\leq\alpha(n).
	\end{equation*}
\end{dfntn} 
Similar to independent random variables (which are $\alpha$-mixing with $\alpha(n)=0$), $\alpha$-mixing allows for quantitative estimates for the error probabilities in the law of large numbers. We will need the following polynomial decay theorem for bounded $\alpha$-mixing sequences, proved by Berbee in \cite{berbeemixing}.
\begin{thrm}\label{amixing}
	Let $p>1$ and $X_j$ be an $\alpha$-mixing sequence of random variables bounded by $1$ such that $\mathds{E}[X_j]=0$ for all $j$. If	
	\begin{equation*}
	\sum_{n\geq 1}n^{p-2}\alpha(n)<+\infty,
	\end{equation*}
	then, setting $S_k=\sum_{j=1}^{k}X_j$, for all $\delta>0$ it holds
	\begin{equation*}
	\sum_{n\geq 1}n^{p-2}\mathbb{P}\left(\sup_{k\geq n}|S_k/k|>\delta\right)<+\infty.
	\end{equation*}
\end{thrm} 
For random fields $\{c_{\xi}\}_{\xi\in\mathcal{Z}^2}$ we have the probabilistic definitions below.
\begin{dfntn}
	Let $\{\tau_z\}_{z\in\mathbb{Z}^2}:\Omega\rightarrow\Omega$ be a measure preserving group action. We say that the random field $\{c_{\xi}\}_{\xi\in\mathcal{Z}^2}$ is
	\begin{enumerate}
		\item[(i)] \emph{stationary}, if $c_{\xi}(\tau_z\w)=c_{\xi+z}(\w)\quad\forall z\in\mathbb{Z}^2$;	
		\item[(ii)] \emph{ergodic}, if it is stationary and $\{\tau_z\}_z$ is ergodic.
		\item[(iii)] \emph{ strongly mixing (in the ergodic sense)}, if it is stationary and	
		\begin{equation*}
		\lim_{|z|\to +\infty}\mathbb{P}(A\cap(\tau_zB))=\mathbb{P}(A)\mathbb{P}(B)\quad\forall A,B\in\mathcal{F};
		\end{equation*}	
		\item[(iv)] \emph{$\alpha$-mixing}, if there exists a sequence $\alpha(n)\to 0$ such that for all sets $I_1,I_2\in\mathcal{Z}^2$ with $\dist(I_1,I_2)\geq n$ we have	
		\begin{align*}
		\sup\{\left|\mathbb{P}(A\cap B)-\mathbb{P}(A)\mathbb{P}(B)\right|:\;A\in\mathcal{F}_{I_1},\,B\in\mathcal{F}_{I_2}\}\leq\alpha(n).
		\end{align*}	
	\end{enumerate}
\end{dfntn}

While for static problems the above notions (i) and (ii) are often enough to prove stochastic homogenization results for variational models (see for example \cite{ACG2,ACR,BCR,BP}), in this minimizing movement setting we make use of mixing properties. More precisely, we require that the random field is $\alpha$-mixing with 
\begin{equation}\label{mixingdecay}
\sum_{n\geq 1}\alpha(n)<+\infty.
\end{equation}
There are stronger notions of mixing in the literature, however we prefer to chose $\alpha$-mixing with a certain decay rate of $\alpha(n)$ rather than some $\phi$-mixing condition since the generalization of $\phi$-mixing conditions to two-dimensional random fields is not trivial and many of them already imply a finite range dependence assumption (see \cite{onmixing}). Moreover, in general $\alpha$-mixing is much weaker than any kind of $\phi$-mixing.

\subsection{Results for deterministic models}
Let us collect some results obtained in the deterministic setting. Within a discrete, deterministic environment, the problem we are interested in has first been studied by Braides, Gelli and Novaga in \cite{BGN} in the case $c_{\xi}(\w)=0$. For coordinate rectangles as initial sets they prove the following:

\begin{thrm}[Braides, Gelli, Novaga]\label{bgn1}
	Let $A_{\e}^0\in\mathcal{A}_{\e}$ be a coordinate rectangle with sides $S_{1,\e},...,S_{4,\e}$. Assume that $A_{\e}^0$ converges in the Hausdorff metric to a coordinate rectangle $A$. Then, up to subsequences, $A_{\e}(t)$ converges locally in time to $A(t)$, where $A(t)$ is a coordinate rectangle with sides $S_i(t)$ such that $A(0)=A$ and any side $S_i$ moves inward with velocity $v_i(t)$ given by
	\begin{equation*}
	v_i(t)
	\begin{cases}
	=\frac{1}{\g}\Big\lfloor\frac{2\g}{L_i(t)}\Big\rfloor &\mbox{if $\frac{2\g}{L_i(t)}\notin\mathbb{N}$,}
	\\
	\\
	\in \frac{1}{\g}\Big[\left(\frac{2\g}{L_i(t)}-1\right),\frac{2\g}{L_i(t)}\Big] &\mbox{if $\frac{2\g}{L_i(t)}\in\mathbb{N}$,}
	\end{cases}
	\end{equation*}	
	where $L_i(t):=\mathcal{H}^1(S_i(t))$ denotes the length of the side $S_i(t)$, until the extinction time when $L_i(t)=0$.
	\\
	Assume in addition that the lengths $L_1^0,L_2^0$ of $A$ satisfy one of the three following conditions (assuming that $L^0_1\leq L^0_2)$:
	\begin{enumerate}
		\item[(i)] $L_1^0,L_2^0>2\g$ (total pinning),
		\item[(ii)] $L_1^0<2\g$ and $L_2^0\leq 2\g$ (vanishing in finite time with shrinking velocity larger than $1/\g$),
		\item[(iii)] $L_1^0<2\g$ such that $2\g/L_1^0\notin\mathbb{N}$ and $L_2^0>2\g$ (partial pinning),
	\end{enumerate}		
	then $A_{\e}(t)$ converges locally in time to $A(t)$ as $\e\to 0$, where $A(t)$ is the unique rectangle with side lengths $L_1(t)$ and $L_2(t)$ solving the following system of ordinary differential equations
	\begin{equation*}
	\begin{cases}
	\frac{d}{dt}L_1(t)=-\frac{2}{\g}\left\lfloor\frac{2\g}{L_2(t)}\right\rfloor,\\
	\\
	\frac{d}{dt}L_2(t)=-\frac{2}{\g}\left\lfloor\frac{2\g}{L_1(t)}\right\rfloor
	\end{cases}
	\end{equation*}
	for almost every $t$ with initial conditions $L_{1}(0)=L_{1}^0$ and $L_2(0)=L_2^0$.
\end{thrm}
It is the aim of this paper to extend these results to small random perturbations of the perimeter. While in \cite{BGN} more general classes of sets are studied, we restrict ourselves to rectangles as the analysis of these sets already contains the main features deriving from randomness. We mention that some effects of periodic perturbations have already been studied in \cite{BrSc,Scilla}. In \cite{BrSc} the authors treat the following type of high-contrast periodicity: Let $N_a,N_b\in\mathbb{N}$ and $N_{ab}=N_a+N_b$. The coefficients $c_{\xi}$ are $N_{ab}$ periodic and on the periodicity cell $0\leq\xi_1,\xi_2<N_{ab}$ they satisfy
\begin{equation*}
c_{\xi}=
\begin{cases}
b &\mbox{if $0\leq\xi_1,\xi_2\leq N_b$,}
\\
a &\mbox{otherwise,}
\end{cases}
\end{equation*}
with weights $a<b$. It is shown that minimizers avoid the $b$-interactions and thus the limit velocity does not depend on $b$ but only on the geometric proportions $N_a,N_b$ of the periodicity cell. It would be interesting to see how random interactions acting on this scale influence the minimizing sets, since without periodicity it might be impossible to take only $a$-interactions. However, in this paper we take the same scaling as the periodic perturbations considered in \cite{Scilla}. These are so called low-contrast perturbations since they vanish when $\e\to 0$. It is shown in \cite{Scilla} that the right scaling to obtain also $b$-interactions is $b-a \sim \e$. More precisely, one has to require that $|b-a|<\frac{\e}{2\gamma}$. Note that this bound agrees with (\ref{pertubationsize}). Hence with $a=1$ this model corresponds to a deterministic version of (\ref{ondual}). In this sense, up to the bound (\ref{pertubationsize}), our Theorem \ref{mainmstatio} generalizes the results of \cite{Scilla} to the most general periodic interactions as well as to the random case. While in the above deterministic setting coefficients with $|b-a|\geq\frac{\e}{2\gamma}$ lead to rectangular interfaces using only $a$-interactions, in the random case it is not clear what happens. We leave this issue as well as the high-contrast case open for future studies. For the interested reader we mention the recent papers \cite{BrCiYi,BrSo}, where the minimizing movements have been studied for other discrete surface-type models.

\section{Homogenized limit motion of a rectangle}
In the sequel we study the case, when the initial data $A_{\e}^0$ is a coordinate rectangle. We further assume for the rest of this paper that
\begin{equation}\label{boundedsetting}
\sup_{\e}\mathcal{H}^1(\partial A_{\e}^0(\w))=C<+\infty.
\end{equation}
This bound implies that any sequence chosen by the minimizing movement has equibounded perimeter. Indeed, by minimality we have
\begin{equation*}
P^{\w}_{\e}(A^{k+1}_{\e}(\w))\leq E^{\w}_{\e}(A_{\e}^{k+1}(\w),A_{\e}^k(\w))\leq E^{\w}_{\e}(A_{\e}^{k}(\w),A_{\e}^{k}(\w))=P^{\w}_{\e}(A_{\e}^{k}(\w)),
\end{equation*}
so that by induction and (\ref{pertubationsize}) we infer
\begin{equation}\label{equiperimeter}
\mathcal{H}^1(\partial A_{\e}^k(\w))\leq 2P^{\w}_{\e}(A_{\e}^k(\w))\leq 2P^{\w}_{\e}(A_{\e}^0(\w))\leq 4\mathcal{H}^1(\partial A_{\e}^0(\w)).
\end{equation}

\subsection{Qualitative behavior}
The main result of this section ensures that coordinate rectangles remain sets of the same type as long as its sides don't degenerate to a point. As we will see later, this is enough to derive the equation of motion at a fixed time $t$. The argument splits into two steps. First we prove that any minimizer must be connected and second, using (\ref{pertubationsize}), we conclude that this component has to be a coordinate rectangle. The idea to prove connectedness is as follows: First we compare the energy with a fast flow of a deterministic functional to conclude that the minimizer must contain a very large rectangle. Then the remaining components are ruled out using the isoperimetric inequality.

\begin{prpstn}\label{stillrectangle}
	Assume that $\{c_{\xi}\}_{\xi}$ fulfills (\ref{pertubationsize}). Let $\eta>0$ and suppose $A_{\e}^k(\w)$ is a coordinate rectangle which has all side lengths greater than $\eta$. Then, for $\e$ small enough, $A_{\e}^{k+1}(\w)$ is again a coordinate rectangle contained in $A_{\e}^k(\w)$.
\end{prpstn}

\begin{proof}
	As explained above we divide the proof into two steps. As the arguments are purely deterministic we drop the $\w$-dependence of the sets.
	\\
	\textbf{Step 1} Connectedness of minimizers
	\\
	We consider the minimizing movement for an auxiliary deterministic functional that turns out to evolve faster. Given $0<\delta<<1$, we define
	\begin{equation*}
	G_{\e}^{\delta}(A,F):=\mathcal{H}^1(\partial A)+\frac{\delta}{\gamma\e}\int_{A\Delta F}d^{\e}_{\infty}(x,\partial F)\,\mathrm{d}x.
	\end{equation*}
	Observe that for any sets $A,B,F\in\mathcal{A}_{\e}$ we have the (in)equalities 
	\begin{align*}
	P_{\e}^{\w}(A\cup B)+P_{\e}^{\w}(A\cap B)\leq &P_{\e}^{\w}(A)+P_{\e}^{\w}(B),
	\\
	\int_{F\Delta (A\cap B)}d_{\infty}^{\e}(x,\partial F)\,\mathrm{d}x+\int_{F\Delta (A\cup B)}d_{\infty}^{\e}(x,\partial F)\,\mathrm{d}x=&\int_{F\Delta A}d_{\infty}^{\e}(x,\partial F)\,\mathrm{d}x
	+\int_{F\Delta  B}d_{\infty}^{\e}(x,\partial F)\,\mathrm{d}x.
	\end{align*}
	The inequality also holds for the standard perimeter, which implies the two general estimates
	\begin{equation}\label{minmax}
	\begin{split}
	&E^{\w}_{\e}(A\cap B,F)+E^{\w}_{\e}(A\cup B,F)\leq E^{\w}_{\e}(A,F)+E^{\w}_{\e}(B,F),
	\\
	&G^{\delta}_{\e}(A\cap B,F)+G^{\delta}_{\e}(A\cup B,F)\leq G^{\delta}_{\e}(A,F)+G^{\delta}_{\e}(B,F).
	\end{split}
	\end{equation}
	Now let $R_{\e}^{\delta}\in\mathcal{A}_{\e}$ be the smallest minimizer of $G^{\delta}_{\e}(\cdot,A_{\e}^k)$ with respect to set inclusion. This is well-defined due to (\ref{minmax}). From the analysis in \cite{BGN} we know that $R_{\e}^{\delta}\subset A^k_{\e}$ is a coordinate rectangle and, denoting by $N_{i,\e}$ the distance between corresponding sides of $R_{\e}^{\delta}$ and $A^k_{\e}$, for $\e$ small enough it holds that
	\begin{equation*}
	\left(\frac{2\gamma}{\delta L_{i,\e}}-1\right)\e\leq N_{i,\e}\leq \left(\frac{2\gamma}{\delta L_{i,\e}}+1\right)\e,
	\end{equation*}  
	where $L_{i,\e}$ denotes the length of the side $S_{i,\e}$ of $A^k_{\e}$. In particular, using (\ref{boundedsetting}), (\ref{equiperimeter}) and the assumptions on the sides of $A^k_{\e}$, we infer the two-sided bound
	\begin{equation}\label{auxvelocity}
	\left(\frac{\gamma}{C\delta}-1\right)\e\leq N_{i,\e}\leq\left(\frac{2\gamma}{\delta\eta}+1\right)\e.
	\end{equation}
	We argue that $R_{\e}^{\delta}\subset A^{k+1}_{\e}$. Assume by contradiction that $R_{\e}^{\delta}\backslash A^{k+1}_{\e}\neq \emptyset$. Since (\ref{auxvelocity}) implies that
	\begin{equation}\label{boundaryestimate}
	d_{\infty}^{\e}(x,\partial A^k_{\e})\geq  \left(\frac{\gamma}{C\delta}-1\right)\e\quad\forall x\in R_{\e}^{\delta},
	\end{equation} 
	using (\ref{minmax}) combined with the fact that both $A^{k+1}_{\e}$ and $R_{\e}^{\delta}$ are minimizers of the corresponding functionals, we obtain
	\begin{align*}
	0\geq &E^{\w}_{\e}(A^{k+1}_{\e},A^k_{\e})-E_{\e}^{\w}(R_{\e}^{\delta}\cup A^{k+1}_{\e},A_{\e}^k)\geq E^{\w}_{\e}(R_{\e}^{\delta}\cap A^{k+1}_{\e},A^k_{\e})-E_{\e}^{\w}(R_{\e}^{\delta},A_{\e}^k)
	\\
	= &\e \left(p_{\e}^{\w}(\partial(R_{\e}^{\delta}\cap A^{k+1}_{\e}))-p_{\e}^{\w}(\partial R_{\e}^{\delta})\right)+\frac{1-\delta}{\gamma\e}\int_{R_{\e}^{\delta}\backslash A^{k+1}_{\e}}d_{\infty}^{\e}(x,\partial A^k_{\e})\,\mathrm{d}x
	\\
	&+\mathcal{H}^1(\partial(R_{\e}^{\delta}\cap A^{k+1}_{\e} ))-\mathcal{H}^1(\partial R_{\e}^{\delta})+\frac{\delta}{\gamma\e}\int_{R^{\delta}_{\e}\backslash A^{k+1}_{\e}}d_{\infty}^{\e}(x,\partial A^k_{\e})\,\mathrm{d}x
	\\
	= &\e \left(p_{\e}^{\w}(\partial (R^{\delta}_{\e}\cap A^{k+1}_{\e}))-p_{\e}^{\w}(\partial R_{\e}^{\delta})\right)
	+\frac{1-\delta}{\gamma\e}\int_{R^{\delta}_{\e}\backslash A^{k+1}_{\e}}d_{\infty}^{\e}(x,\partial A^k_{\e})\,\mathrm{d}x
	\\
	&+G^{\delta}_{\e}(R^{\delta}_{\e}\cap A^{k+1}_{\e},A^k_{\e})-G_{\e}^{\delta}(R_{\e}^{\delta},A^k_{\e})
	\\
	\geq &\e \left(p_{\e}^{\w}(\partial (R^{\delta}_{\e}\cap A^{k+1}_{\e} ))-p_{\e}^{\w}(\partial R_{\e}^{\delta})\right)+\frac{1-\delta}{\gamma\e}\int_{R_{\e}^{\delta}\backslash A^{k+1}_{\e}}d_{\infty}^{\e}(x,\partial A^k_{\e})\,\mathrm{d}x,
	\end{align*}
	where we used several times that $R_{\e}^{\delta}\subset A^k_{\e}$ to simplify the symmetric differences. In combination with (\ref{boundaryestimate}), for $\delta\leq\frac{1}{2}$ the last estimate yields
	\begin{equation}\label{iso1}
	\left(\frac{1}{2C\delta}-\frac{1}{2\gamma}\right)|R^{\delta}_{\e}\backslash A_{\e}^{k+1}|\leq \e \Big(p_{\e}^{\w}(\partial R_{\e}^{\delta})- p_{\e}^{\w}(\partial (R^{\delta}_{\e}\cap A^{k+1}_{\e}))\Big).
	\end{equation} 
	In order to use this inequality, we need to analyze which boundary contributions cancel in the last difference. Given $\xi=\frac{i+j}{2}\in\mathcal{Z}^2$ we distinguish two exhaustive cases:
	\begin{itemize}
		\item [(i)] $i\in R_{\e}^{\delta},\,j\notin R_{\e}^{\delta}$: If $i\in A^{k+1}_{\e}$ we have $i\in R_{\e}^{\delta}\cap A^{k+1}_{\e}$ and $j\notin R_{\e}^{\delta}\cap A^{k+1}_{\e}$ which implies $\xi\in\partial(R_{\e}^{\delta}\cap A^{k+1}_{\e})$ and thus this contribution cancels. Otherwise $i\notin A^{k+1}_{\e}$ and consequently $\xi\in\partial (R_{\e}^{\delta}\backslash A^{k+1}_{\e})$;
		\item[(ii)]
		$i\in R_{\e}^{\delta}\cap A^{k+1}_{\e},\,j\notin R_{\e}^{\delta}\cap A^{k+1}_{\e}$: If $j\notin R_{\e}^{\delta}$, then $\xi\in\partial R_{\e}^{\delta} $ and the contribution cancels, while $j\in R_{\e}^{\delta}$ yields $j\notin A^{k+1}_{\e}$ and therefore $\xi\in\partial (R_{\e}^{\delta}\backslash A^{k+1}_{\e})$.
	\end{itemize}
	From those two cases and (\ref{pertubationsize}) we infer that 
	\begin{equation*}
	\e p_{\e}^{\w}(\partial R_{\e}^{\delta})-\e p_{\e}^{\w}(\partial (R^{\delta}_{\e}\cap A^{k+1}_{\e}))\leq \frac{1}{4\gamma}\e\mathcal{H}^1(\partial (R_{\e}^{\delta}\backslash A^{k+1}_{\e})).
	\end{equation*}
	Since for all sets $A\in\mathcal{A}_{\e}$ we have the reverse isoperimetric inequality $\e\mathcal{H}^1(\partial A)\leq 4|A|$, we can put together the last inequality and (\ref{iso1}) to deduce
	\begin{equation*}
	\left(\frac{1}{2C\delta}-\frac{1}{2\gamma}\right)|R_{\e}^{\delta}\backslash A_{\e}^{k+1}|\leq \frac{1}{\gamma}|R_{\e}^{\delta}\backslash A^{k+1}_{\e}|.
	\end{equation*}
	Choosing $\delta$ small enough this yields a contradiction. Hence we proved that $R_{\e}^{\delta}\subset A^{k+1}_{\e}$ for $\delta$ small enough.
	
	Next we rule out any other connected component except the one containing $R_{\e}^{\delta}$. Note that estimate (\ref{auxvelocity}) implies
	\begin{equation}\label{auxvelocity3}
	d_{\infty}^{\e}(x,\partial A^k_{\e})\leq (\frac{2\gamma}{\delta\eta}+1)\e\quad\quad\forall x\in A^k_{\e}\backslash R_{\e}^{\delta}.
	\end{equation}
	Consider a connected component $A$ of $A^{k+1}_{\e}$ not containing $R_{\e}^{\delta}$. We set $A^{\prime}=A_{\e}^{k+1}\backslash A$. Due to (\ref{boundedsetting}) and (\ref{auxvelocity3}) it holds that $|A^k_{\e}\cap A|\leq |A_{\e}^k\backslash R_{\e}^{\delta}|\leq C_{\delta,\eta}\e$. Hence, for $\e$ small enough,
	\begin{align*}
	E^{\w}_{\e}(A^{k+1}_{\e},A^k_{\e})-E^{\w}_{\e}(A^{\prime},A^k_{\e})\geq& (1-\frac{1}{4\gamma}\e)\mathcal{H}^1(\partial A)-\frac{1}{\gamma\e}\int_{A^{k}_{\e}\cap A}d_{\infty}^{\e}(x,\partial A^k_{\e})\,\mathrm{d}x
	\\
	\geq& \frac{\mathcal{H}^1(\partial A)}{2}-\Big(\frac{2}{\delta\eta}+\frac{1}{\gamma}\Big)|A^k_{\e}\cap A|\geq \frac{\mathcal{H}^1(\partial A)}{2}-|A^k_{\e}\cap A|^{\frac{1}{2}}
	\\
	\geq&\frac{\mathcal{H}^1(\partial A)}{2}-|A|^{\frac{1}{2}}\geq \frac{1}{2}(1-\frac{1}{\sqrt{\pi}})\mathcal{H}^1(\partial A)>0,
	\end{align*} 
	where we used the two-dimensional isoperimetric inequality. This contradicts the minimality of $A^{k+1}_{\e}$ and we conclude that $A^{k+1}_{\e}$ has exactly one connected component.
	\\
	\textbf{Step 2} Reduction to coordinate rectangles
	\\
	First note that if we replace an arbitrary set $A\in\mathcal{A}_{\e}$ by the set $A\cap A_{\e}^k$ we strictly reduce the energy if the sets are not equal. To see this, we observe that
	\begin{align}\label{cutboundary}
	E^{\w}_{\e}(A,A_{\e}^k)-E^{\w}_{\e}(A\cap A_{\e}^k,A_{\e}^k)\geq & \frac{1}{\g\,\e}\int_{A\backslash A_{\e}^k}d_{\infty}^{\e}(x,\partial A_{\e}^k)\,\mathrm{d}x+P_{\e}^{\w}(A)-P_{\e}^{\w}(A\cap A_{\e}^k)\nonumber
	\\
	\geq& \frac{|A\backslash A_{\e}^k|}{\g}+P_{\e}^{\w}(A)-P_{\e}^{\w}(A\cap A_{\e}^k).
	\end{align}
	Again we need to analyze which interactions cancel due to the random perimeter difference. As $A_{\e}^k$ is a coordinate rectangle, by elementary geometric considerations one can prove that $\mathcal{H}^1(\partial A)\geq \mathcal{H}^1(\partial(A\cap A_{\e}^k))$. On the other hand, reasoning similar to the lines succeeding (\ref{iso1}) one can show that all random interactions cancel except those coming from $\partial (A\backslash A_{\e}^k)$. In case this set is non-empty, by (\ref{pertubationsize}) we conclude that (\ref{cutboundary}) can be further estimated via the strict inequality
	\begin{align*}
	E^{\w}_{\e}(A,A_{\e}^k)-E^{\w}_{\e}(A\cap A_{\e}^k,A_{\e}^k)> &\frac{|A\backslash A_{\e}^k|}{\g}-\frac{1}{4\gamma}\e\mathcal{H}^{d-1}(\partial(A\backslash A_{\e}^k))\geq 0,
	\end{align*}
	where we used again the reverse isoperimetric inequality in $\mathcal{A}_{\e}$. Whenever $A$ is a minimizer we get a contradiction which shows that $\partial(A_{\e}^{k+1}\backslash A_{\e}^k)=\emptyset$, or equivalently $A_{\e}^{k+1}\subset A_{\e}^k$.
	
	To conclude we assume by contradiction that $A_{\e}^{k+1}$ is not a coordinate rectangle. Consider then the minimal coordinate rectangle $R$ containing $A_{\e}^{k+1}$ (see Figure \ref{minrectangle}). Then again by elementary geometric arguments it holds that $\mathcal{H}^1(\partial A_{\e}^{k+1}\backslash \partial R)\geq\mathcal{H}^{1}(\partial R\backslash\partial A_{\e}^{k+1})$. 
	As $R\subset A_{\e}^k$ by the previous argument, using (\ref{pertubationsize}) the difference of the energies can be estimated by
	\begin{align}\label{rectangle}
	0\geq& E^{\w}_{\e}(A_{\e}^{k+1},A_{\e}^k)-E^{\w}_{\e}(R,A_{\e}^k)\geq
	\frac{1}{\g}|R\backslash A_{\e}^{k+1}|+P_{\e}^{\w}(A_{\e}^{k+1})-P_{\e}^{\w}(R)\nonumber
	\\
	>&\frac{1}{\g}|R\backslash A_{\e}^{k+1}|+(1-\frac{1}{4\gamma}\e)\mathcal{H}^1(\partial A_{\e}^{k+1}\backslash \partial R)-(1+\frac{1}{4\gamma}\e)\mathcal{H}^1(\partial R\backslash \partial A_{\e}^{k+1})\nonumber
	\\
	\geq&\frac{1}{\g}|R\backslash A_{\e}^{k+1}|-\frac{1}{2\gamma}\e\mathcal{H}^1(\partial R\backslash \partial A_{\e}^{k+1}) .
	\end{align} 
	\begin{figure}[h]
		\centerline{\includegraphics [trim=7cm 9cm 5cm 7cm, scale=0.3]{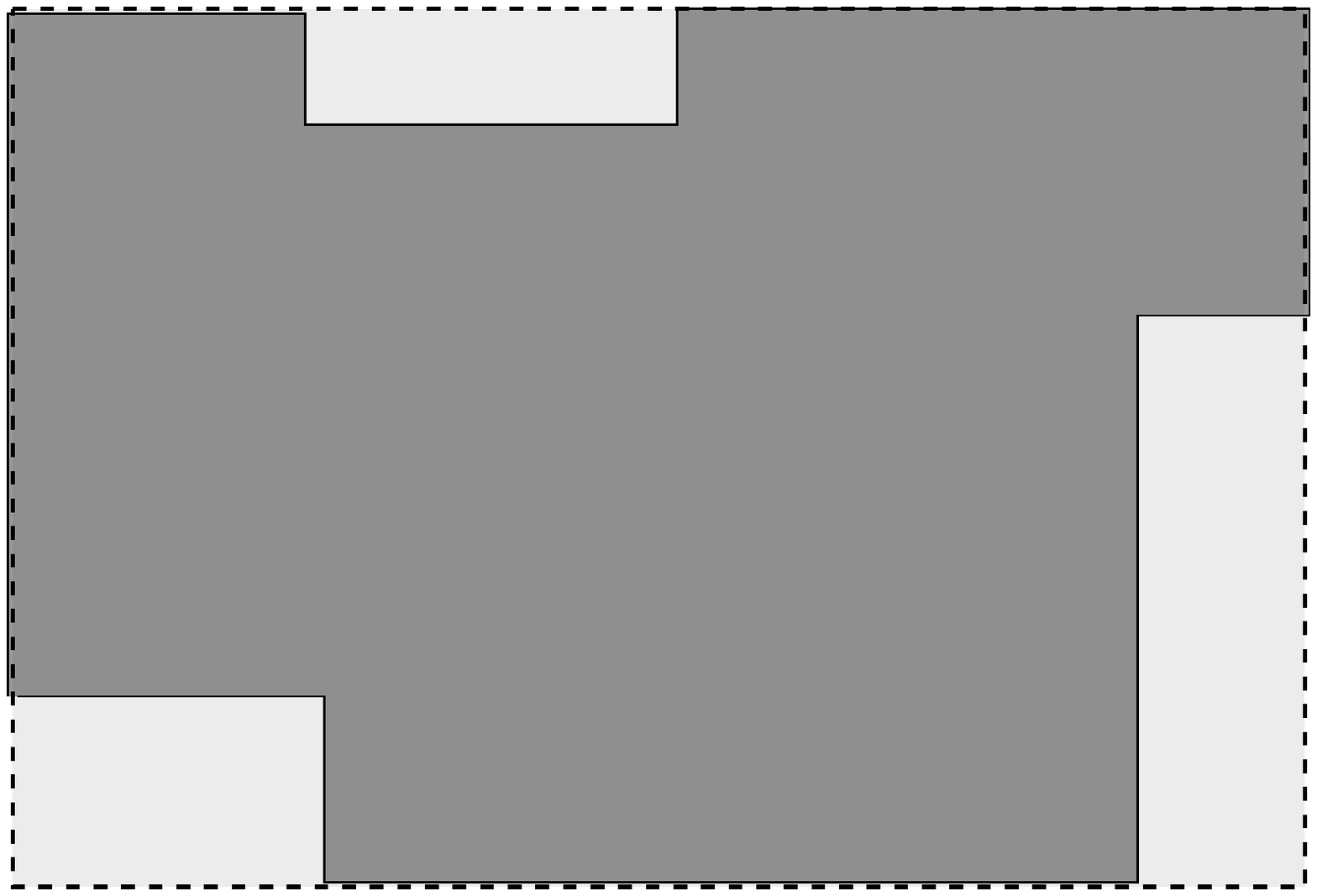}}
		\caption[Minimal rectangularizations as competitors]{The minimal coordinate rectangle $R$ containing a connected set $A_{\e}\in\mathcal{A}_{\e}$. It has less (or equal) perimeter and (\ref{measuregain}) holds.} \label{minrectangle}
		\begin{picture}(0,0)
		\put(15,100){$A_{\e}$}\put(45,80){R}
		\end{picture}
	\end{figure}
	Due to the strict inequality in (\ref{rectangle}) we conclude the proof as soon as we show that 
	\begin{equation}\label{measuregain}
	|R\backslash A_{\e}^{k+1}|\geq \frac{\e}{2}\mathcal{H}^1(\partial R\backslash\partial A_{\e}^{k+1}).
	\end{equation}
	Arguing locally on each connected component of $\partial R\backslash\partial A_{\e}^{k+1}$ we can assume that $\partial R\backslash\partial A_{\e}^{k+1}$ has one connected component $R_c$. We distinguish between the cases where $R_c$ contains only horizontal or vertical segments and where it has both, that is to say at a corner of $R$. If the component has only one type of segment then obviously $|R\backslash A_{\e}^{k+1}|\geq \e \mathcal{H}^1(R_c)$. In the other case let us assume that the horizontal segment $R_{c,x}$ is the larger one. Then $|R\backslash A_{\e}^{k+1}|\geq \e \mathcal{H}^1(R_{c,x})$ which also implies (\ref{measuregain}). This completes the proof.	
\end{proof}

\begin{rmrk}\label{genericbound}
	{\rm We want to comment on the $L^{\infty}$-bound (\ref{pertubationsize}): First note that is optimal in the sense that if it is violated, with positive probability there could be defects at the corners of the rectangle. On the other hand, our argument for proving connectedness can be extended with a slight effort to a generic $L^{\infty}$-bound on the random field. In order to determine the shape one then would need to use probabilistic arguments. We strongly believe (even though we did not check the argument in detail), that the law of large numbers implies that the minimizer must have the same deterministic perimeter as the minimal coordinate rectangle containing it. Moreover, the bulk term yields a control on the deviation from this minimal rectangle. However deviations can exist and that causes difficulties. In order to apply an inductive argument (which is necessary also for connectedness), one needs to control the deviation from a rectangle. However we are not able to rule out that deviations grow with the number of time steps. 
	}	
\end{rmrk}

\subsection{Computation of the velocity}
As a next step we derive a precise formula for the velocity of the discrete motion. We follow \cite{BGN} and express the functional to be minimized by the distance from each side of the optimal rectangle to the corresponding side of the previous set $A_{\e}^k(\w)$. Let $A_{\e}^{k+1}(\w)$ be a minimizer. To reduce notation, we let $s_{i,\e}$ and $s^{\prime}_{i,\e}$ $(i=1,\dots,4)$ be the sides of $A_{\e}^k(\w)$ and $A_{\e}^{k+1}(\w)$ respectively and set $l_{i,\e}=\mathcal{H}^1(s_{i,\e})$. We define $N^{k+1}_{i,\e}\e$ as the distance from the side $s_{i,\e}$ to the side $s^{\prime}_{i,\e}$. It can be easily shown that $A_{\e}^{k+1}(\w)$ must contain the center of the previous rectangle $A_{\e}^k(\w)$. Rewriting the functional $E^{\w}_{\e}(A,A_{\e}^k)$ in terms of the four integer numbers $N^{k+1}_{i,\e}$, we obtain that these are minimizers of the function $\tilde{f}^{\w}_{\e}:\mathbb{N}^4\to\R$ defined by
\begin{align*}
\tilde{f}_{\e}^{\w}(N):=&\sum_{i=1}^{4}(l_{i,\e}-2N_i\e)+\sum_{i=1}^4p^{\w}_{\e}(s_{i,\e}+N_{i}\e v_i)\e-\e^2e^{per}_{\e}+\frac{\e}{\g}\sum_{i=1}^4\sum_{n=1}^{N_i}l_{i,\e}n-\e^2e_{\e}^{bulk}
\\
=&\e\sum_{i=1}^{4}\left((\frac{l_{i,\e}}{\e}-2N_i)+p^{\w}_{\e}(s_{i,\e}+N_{i}\e v_i)+\frac{(N_i+1)N_i\,l_{i,\e}}{2\g}\right)-\e^2(e^{per}_{\e}+e_{\e}^{bulk}),
\end{align*}
where $v_i\in\{\pm e_1,\pm e_2\}$ denotes the vector representing the inward motion of each side and the error terms $e_{\e}^{per},\,e_{\e}^{bulk}$ account for the fact that we neglect the shrinking effect on the random part of the energy and that we count twice the bulk part in the corners (one time with the wrong distance). For these errors we have the following bounds:
\begin{equation}\label{errors}
|e_{\e}^{per}|\leq \frac{2}{\gamma}\max_i N_i,\quad\quad
|e_{\e}^{bulk}|\leq \frac{4}{\gamma}(\max_i N_i)^3.
\end{equation}
We argue that the error terms are negligible as $\e\to 0$. To this end we show that $\max_i N_{i,\e}$ is equibounded with respect to $\e$ as long as $l_{i,\e}\geq\eta>0$ for some $\eta>0$. Indeed, suppose without loss of generality that $N^*:=\max_i N_i$ corresponds to the right vertical side $s_{i,\e}$. Let us denote by $P$ the center of $A_{\e}^k(\w)$. Then, for $\e$ small enough, one can easily prove that
\begin{equation*}
\{x\in A_{\e}^k(\w):\;\frac{N}{2}\e\leq\dist(x,s_{i,\e})\leq N\e,\;|\langle x-P,e_2\rangle|\leq \frac{\eta}{4}\}\subset A_{\e}^{k}(\w)\backslash A_{\e}^{k+1}(\w).
\end{equation*}
Thus for the bulk term we obtain the lower bound
\begin{equation*}
\frac{1}{\g\e}\int_{A_{\e}^k(\w)\backslash A_{\e}^{k+1}(\w)}d_{\infty}^{\e}(x,\partial A_{\e}^k)\,\mathrm{d}x \geq \frac{\min\{\frac{\eta}{4},\frac{N}{2}\e\}}{\g}\frac{N\eta}{4}.
\end{equation*} 
Using (\ref{pertubationsize}) and (\ref{boundedsetting}), for $\e$ small enough, we deduce a lower bound for the random perimeter via
\begin{align*}
P^{\w}_{\e}(A_{\e}^{k+1}(\w))&\geq P^{\w}_{\e}(A_{\e}^k(\w))-8N\e-\frac{1}{2\g}\e\mathcal{H}^1(\partial A_{\e}^k(\w))-\frac{2}{\gamma}\e^2N
\\
&\geq E^{\w}_{\e}(A_{\e}^k(\w),A_{\e}^k(\w))-\e\left(9N+\frac{C}{2\gamma}\right).
\end{align*}
Assuming that $N\geq \frac{C}{2\gamma}$, we infer that such $N$ can't yield a minimizer as soon as
\begin{equation}\label{coercivity}
-10N\e+\frac{\min\{\frac{\eta}{4},\frac{N}{2}\e\}}{\g}\frac{N\eta}{4}>0.
\end{equation}
From (\ref{coercivity}) one can easily deduce that $N$ has to be bounded when $\e\to 0$.

It follows from (\ref{errors}) that, asymptotically, we can instead minimize the functional
\begin{equation}\label{minvelocity}
{f}_{\e}^{\w}(N)=\sum_{i=1}^4\left(-2N_i+p^{\w}_{\e}(s_{i,\e}+N_{i}\e v_i)+\frac{1}{2\g}(N_i+1)N_i\,l_{i,\e}\right),
\end{equation}
provided that the minimizer of the limit is unique. In particular, as in \cite{BGN}  each side moves independently from the remaining ones. More precisely, we have to study the minimizers of the one-dimensional random function
\begin{equation}\label{sidevelocity}
v_{i,\e}^{\w}(N):=-2N+p_{\e}^{\w}(s_{i,\e}+N\e v_i)+\frac{1}{2\g}(N+1)N\,l_{i,\e}.
\end{equation}
The asymptotic behavior of the stochastic term in (\ref{sidevelocity}) is more involved since the segment $s_{i,\e}$ can vary along infinitely many different lattice positions as $\e\to 0$. Thus a direct application of Birkhoff's ergodic theorem to prove the existence of a limit is not possible. Indeed, in what follows we will show the existence of a stationary, ergodic system of perturbations where for at least one side $s_{i,\e}$ the term $p_{\e}^{\w}(s_{i,\e}+N\e v_i)$ does not converge with probability $1$. We will come back to this example in Section \ref{lastexample}, where we indicate how to treat this case making some further assumptions.
\begin{xmpl}\label{nonconv}
	Let $\g=1$ and let $X_i,\,i\in\mathbb{Z}$ be a sequence of non-constant independent and identically distributed random variables on a probability space $(\Omega,\mathcal{F},\mathbb{P})$ equipped with a measure-preserving ergodic map $\tau:\Omega\rightarrow\Omega$ such that $X_k(\w)=X_0(\tau^k\w)$  (this setup can be realized on a suitable product space with the shift operator). Moreover assume that $\|X_i\|_{\infty}<\frac{1}{4}$ and set $c_{\xi}(\w)=X_{\lfloor\xi_1\rfloor}(\w)$. Then $c_{\xi}$ is a stationary, ergodic random field. If the initial coordinate rectangles $A_{\e}^0$ converge in the Hausdorff metric to a coordinate rectangle $A^0$, then for at least one of the vertical sides we have that, for all $N\in\mathbb{N}$,
	\begin{equation*}
	\mathbb{P}\left(\{\w:\;\lim_{\e\to 0}v_{i,\e}^{\w}(N)\text{ exists}\}\right)=0.
	\end{equation*}
\end{xmpl} 

\begin{proof}
	Let $\e_n\to 0$. Note that for at least one vertical side the $x$-component of $\frac{s_{i,\e_n}}{\e_n}$ diverges to $\pm\infty$. Then so does the $x$-component of $\frac{s_{i,\e_n}}{\e_n}+Nv_i$. Without loss of generality we assume that these $x$-components form a sequence of positive numbers $\{k_n+\frac{1}{2}\}_n\to +\infty$ with $k_n\in\mathbb{N}$. Passing to a subsequence (not relabeled) we can assume that this sequence is monotone increasing. Since $l_{i,\e_n}$ converges to the vertical side length $l_i$ of $A^0$, we only have to take into account the random term. Since $A_{\e_n}^0\in \mathcal{A}_{\e_n}$, we have 
	\begin{equation*}
	p_{\e_n}^{\w}(s_{i,\e_n}+N\e_n v_i)=X_{k_n}(\w)l_{i,\e_n}.
	\end{equation*}
	Since $l_{i,\e_n}$ converges to $l_i\neq 0$, the asymptotic behavior of $p_{\e_n}^{\w}(s_{i,\e_n}+N\e_nv_i)$ is characterized by $X_{k_n}(\w)$. Since these variables are non-constant and independent, it follows from Kolmogorov's 0-1 law that 
	\begin{equation*}
	\mathbb{P}(\{\w:\;\lim_nX_{k_n}(\w)\text{ exists}\})=0.
	\end{equation*}
	Moreover we can define the measure preserving group action $\tau_z:\Omega\to\Omega$ as	
	\begin{equation*}
	\tau_z\w:=\tau^{z_1}\w.
	\end{equation*}
	From the construction of the random field, it follows immediately that $\{c_{\xi}\}_{\xi}$ is stationary. By assumption the group action is ergodic, too.	
\end{proof}

Despite the negative result of the previous example, we now show that even in the worst case the term $p_{\e}^{\w}(s_{i,\e}+N\e v_i)$ doesn't influence the range of possible minimizers too much. Indeed, by (\ref{pertubationsize}) we have
\begin{equation*}
\sup_{N,N^{\prime}}|p_{\e}^{\w}(s_{i,\e}+N\e v_i)-p_{\e}^{\w}(s_{i,\e}+N^{\prime}\e v_i)|\leq \frac{1}{2\g}l_{i,\e},
\end{equation*}
while (one of) the integer minimizers for the polynomial $P(x)=-2x+\frac{l_{i,\e}}{2\g}(x+1)x$ is given by $x^*=\lfloor\frac{2\g}{l_{i,\e}}\rfloor$. We  deduce the estimate
\begin{align*}
|P(x^*\pm 2)-P(x^*)|=\begin{cases}
\frac{3l_{i,\e}}{\g}-4+\frac{2l_{i,\e}}{\g}x^*\geq\frac{l_{i,\e}}{\g}, \\
\frac{l_{i,\e}}{\g}+4-\frac{2l_{i,\e}}{\g}x^*\geq \frac{l_{i,\e}}{\g}.
\end{cases}
\end{align*}
We infer that for minimizing $v_{i,\e}^{\w}$ we need only to consider three values, namely
\begin{equation}\label{only3}
\min_N v_{i,\e}^{\w}(N)=\min\{v_{i,\e}^{\w}(x^*),\,v_{i,\e}^{\w}(x^*+1),\,v_{i,\e}^{\w}(x^*-1)\}.
\end{equation}

Thus the randomness can only cause one additional jump forwards or backwards. In order to obtain the convergence we need a stronger form of independence than ergodicity that is preserved on one-dimensional sections of $\mathbb{Z}^2$. It turns out that the $\alpha$-mixing condition introduced in (\ref{mixingdecay}) is enough. Indeed, we have the following crucial result:
\begin{prpstn}\label{allconvergence}
	Assume that the random field $\{c_{\xi}\}_{\xi}$ is stationary and $\alpha$-mixing such that (\ref{mixingdecay}) holds and set $\mu:=\mathds{E}[c_{\xi}]$. Let $\e_j\downarrow 0$. There exists a set $\Omega^{\prime}\subset\Omega$ of full probability (independent of the particular sequence $\e_j$) such that for every $\w\in\Omega^{\prime}$ and every sequence of sides $\{S_j\}_{j\in\mathbb{N}}$ such that $S_j$ converges in the Hausdorff metric to a segment $S$, we have
	\begin{equation*}
	\lim_jp^{\w}_{\e_j}(S_j)=\mathcal{H}^1(S)\mu.
	\end{equation*}
\end{prpstn}

\begin{proof}
	We assume that the side is a vertical side, the case of horizontal sides works the same way with another set of full measure. Moreover it is not restrictive to consider the case $\mu=0$. To reduce notation we let $[x]^*:=\lfloor x\rfloor +\frac{1}{2}$. Given $q\in\mathbb{Q}\cap(0,+\infty)$ we define the following sequences of random variables:	
	\begin{equation*}
	X_n^{q,\pm}(\w):=\sup_{k\geq qn}\left|\frac{1}{2k+1}\sum_{l=-k}^{k}c_{([\pm n]^*,l)}(\w)\right|.
	\end{equation*}	
	Given $\delta>0$, by stationarity and an elementary fact about average sums we have
	\begin{align}
	\mathbb{P}\left(|X^{q,\pm}_n|>\delta\right)&=\mathbb{P}\bigg(\sup_{k\geq qn}\bigg|\frac{1}{2k+1}\sum_{l=-k}^{k}c_{([0]^*,l)}\bigg|>\delta\bigg)\nonumber
	\\
	&\leq \mathbb{P}\bigg(\sup_{k\geq qn}\bigg|\frac{1}{k+1}\sum_{l=0}^{k}c_{([0]^*,l)}\bigg|>\delta\bigg)+\mathbb{P}\bigg(\sup_{k\geq qn}\bigg|\frac{1}{k}\sum_{l=1}^{k}c_{([0]^*,-l)}\bigg|>\delta\bigg).
	\label{statio}
	\end{align}	
	Upon rescaling $c_{\xi}$ we can apply Theorem \ref{amixing} with $p=2$ to the two bounded and $\alpha$-mixing sequences $\{c_{([0]^*,l)}\}_{l\in\mathbb{N}}$ and $\{c_{([0]^*,-l)}\}_{l\in\mathbb{N}}$ and deduce from (\ref{statio}) that
	\begin{align*}
	\sum_{n\geq 1}\mathbb{P}(|X_n^{q,\pm}|>\delta)
	&\leq\sum_{n\geq 1}\mathbb{P}\bigg(\sup_{k\geq qn}\bigg|\frac{1}{k+1}\sum_{l=0}^{k}c_{([0]^*,l)}\bigg|>\delta\bigg)+\mathbb{P}\bigg(\sup_{k\geq qn}\bigg|\frac{1}{k}\sum_{l=1}^{k}c_{([0]^*,-l)}\bigg|>\delta\bigg)
	\\
	&\leq \lceil q^{-1}\rceil\sum_{i\geq 1}\mathbb{P}\bigg(\sup_{k\geq i}\bigg|\frac{1}{k+1}\sum_{l=0}^{k}c_{([0]^*,l)}\bigg|>\delta\bigg)+\mathbb{P}\bigg(\sup_{k\geq i}\bigg|\frac{1}{k}\sum_{l=1}^{k}c_{([0]^*,-l)}\bigg|>\delta\bigg)
	<+\infty.
	\end{align*} 
	Hence by the Borel-Cantelli Lemma there exists a set of full probability $\Omega^q$ such that both $X_n^{q,+}$ and $X_n^{q,-}$ converge to $0$ pointwise on $\Omega^q$. We set $\Omega^{\prime\prime}:=\bigcap_{q}\Omega^q$.
	
	Next we check that we can relate the random length of the side $S_j$ to one of the random variables $X^{q,\pm}_n$. Let $S_j$ converge to a segment $S$ in the Hausdorff metric and denote by $x\in\R$ the $x$-coordinate of $S$. We start with the case $x>0$. Fix $\beta>0$ and let $x_j\in\mathbb{Z}+\frac{1}{2}$ be the $x$-component of $S_j/\e_j$. Then there exists $j_0=j_0(\beta)$ such that for all $j\geq j_0$ we have $x+\beta\geq \e_jx_j$ and $\e_j \#\{\xi\in S_j/\e_j\cap\mathcal{Z}^2\}\geq \mathcal{H}^1(S)-\beta$. For such $j$ we infer that
	\begin{equation*}
	\#\{\xi\in \frac{S_j}{\e_j}\cap\mathcal{Z}^2\}\geq\frac{\mathcal{H}^1(S)-\beta}{x+\beta}x_j.
	\end{equation*}	
	For $\beta$ small enough, there exists $q\in\mathbb{Q}$ such that $\frac{\mathcal{H}^1(S)-\beta}{x+\beta}>3q>0$. Now for every $j$ we let $n_j\in\mathbb{N}$ satisfying $[n_j]^*=x_j$ (we may assume that $x_j>0$ for all $j$). Then
	\begin{equation}\label{goodratio}
	\#\{\xi\in \frac{S_j}{\e_j}\cap\mathcal{Z}^2\}> 3qn_j.
	\end{equation}	
	Let us first assume that $S=\{x\}\times\frac{1}{2}[-\mathcal{H}^1(S),\mathcal{H}^1(S)]$ is a centered side. Then for $j$ large enough it holds
	\begin{equation}\label{lenghtclose}
	\#\Big(\Big\{\xi\in \frac{S_j}{\e_j}\cap\mathcal{Z}^2\Big\}\Delta\Big\{\xi=(x_j,l)\in\mathcal{Z}^2:\;|l|\leq\frac{\mathcal{H}^1(S_j)}{2\e_j}\Big\}\Big)\leq \frac{\beta}{\e_j},
	\end{equation}
	so that by (\ref{pertubationsize}) we have	
	\begin{align}
	\bigg|\sum_{\e_j\xi\in S_j}\e_jc_{\xi}(\w)\bigg|&=\mathcal{H}^1(S_j)\bigg|\frac{1}{\#\{\e_j\xi\in S_j\}}\sum_{\e_j\xi\in S_j}c_{\xi}(\w)\bigg|\nonumber
	\\
	&\leq C\beta +C\bigg|\frac{1}{\#\{\e_j\xi\in S_j\}}\sum_{2\e_j|l|\leq\mathcal{H}^1(S_j)}c_{(x_j,l)}(\w)\bigg|\leq C\beta +CX_{n_j}^{q,+}(\w),\label{withq}
	\end{align}
	where we used that $\lfloor\mathcal{H}^1(S_j)/(2\e_j)\rfloor\geq qn_j$ for all but finitely many $j$ by (\ref{goodratio}). Since $\beta>0$ is arbitrary and $X_{n_j}^{q,+}(\w)\to 0$ for all $\w\in\Omega^{\prime\prime}$ we conclude in this special case.
	
	Now assume that $S=\{x\}\times\frac{1}{2}[y-\mathcal{H}^1(S),y+\mathcal{H}^1(S)]$ with $y>0$ (the other case is similar). We aim to transfer the variables pointwise with the help of the group action. For $\beta>0$ and $q$ fixed as above, we define the events
	\begin{equation*}
	\mathcal{Q}_N:=\left\{\w\in\Omega:\;\forall n\geq \frac{N}{2}\text{ it holds }|X_n^{q,+}(\w)|\leq\beta\right\}.
	\end{equation*}
	By the arguments hitherto we know that the function $\mathds{1}_{\mathcal{Q}_N}$ converges to $\mathds{1}_{\Omega}$ almost surely. Let us denote by $\mathcal{J}_{e_2}$ the (maybe non-trivial) $\sigma$-algebra of invariant sets for the measure preserving map $\tau_{e_2}$. Fatou's lemma for the conditional expectation yields
	\begin{equation*}
	\mathds{1}_{\Omega}=\mathds{E}[\mathds{1}_{\Omega}|\mathcal{J}_{e_2}]\leq\liminf_{N\to +\infty}\mathds{E}[\mathds{1}_{\mathcal{Q}_N}|\mathcal{J}_{e_2}].
	\end{equation*} 
	Hence we know that, given $\delta>0$, almost surely we find $N_0=N_0(\w,\delta)$ such that
	\begin{equation*}
	1\geq\mathds{E}[\mathds{1}_{\mathcal{Q}_{N_0}}|\mathcal{J}_{e_2}](\w)\geq 1-\delta.
	\end{equation*} 
	Due to Birkhoff's ergodic theorem, almost surely, there exists $n_0=n_0(\w,\delta)$ such that, for any $m\geq \frac{1}{2}n_0$,
	\begin{equation*}
	\left|\frac{1}{m}\sum_{i=1}^m\mathds{1}_{\mathcal{Q}_{N_0}}(\tau_{i {e_2}}\w)-\mathds{E}[\mathds{1}_{\mathcal{Q}_{N_0}}|\mathcal{J}_{e_2}](\w)\right|\leq\delta.
	\end{equation*}
	Note that the set we exclude will be a countable union of null sets (depending only on the sequences $X_n^{q,\pm}$ and rational $\beta$). With a slight abuse of notation we still call the smaller set $\Omega^{\prime\prime}$.

	We now fix $\w\in\Omega^{\prime\prime}$. For $m\geq\max\{n_0(\w,\delta),N_0(\w,\delta)\}$ we denote by $R$ the maximal integer such that for all $i=m+1,\dots,m+R$ we have $\tau_{ie_2}(\w)\notin \mathcal{Q}_{N_0}$. In order to bound $R$ let $\tilde{m}$ be the number of non-zero elements in the sequence $\{\mathds{1}_{\mathcal{Q}_{N_0}}(\tau_{ie_2}(\w))\}_{i=1}^{m}$. By definition of $R$ we have
	\begin{equation*}
	\delta\geq\left|\frac{\tilde{m}}{m+R}-\mathbb{E}[1_{\mathcal{Q}_{N_0}}|\mathcal{J}_{e_2}](\w)\right|=\left|1-\mathbb{E}[1_{\mathcal{Q}_{N_0}}|\mathcal{J}_{e_2}](\w)+\frac{\tilde{m}-m-R}{m+R}\right|\geq \frac{R+m-\tilde{m}}{m+R}-\delta.
	\end{equation*}
	Since $m-\tilde{m}\geq 0$ and without loss of generality $\delta\leq\frac{1}{4}$, this provides an upper bound by $R\leq 4m\delta$. Thus for an arbitrary $m\geq\max\{n_0(\w,\delta),N_0(\w,\delta)\}$ and $\tilde{R}=6m\delta$ we find $l_m\in [m+1,m+\tilde{R}]$ such that $\tau_{l_me_2}(\w)\in \mathcal{Q}_{N_0}$. Then we have for all $n\geq \frac{N_0}{2}$ that
	\begin{equation}\label{betaestimate}
	|X_n^{q,+}(\tau_{l_me_2}\w)|\leq \beta.
	\end{equation}
	For $j$ large enough we have $\lfloor y/\e_j\rfloor\geq \max\{n_0(\w,\delta),N_0(\w,\delta)\}$ so that there exists $l_j\in\mathbb{N}$ satisfying (\ref{betaestimate}) and moreover
	\begin{equation}\label{almostshifted}
	|\lfloor{y}/{\e_j}\rfloor-l_j|\leq 6\delta\lfloor{y}/{\e_j}\rfloor.
	\end{equation}
	In addition we can assume that $|\mathcal{H}^1(S)-\mathcal{H}^1(S_j)|\leq\beta$. Note that (\ref{almostshifted}) is the analogue of (\ref{lenghtclose}). Thus from (\ref{betaestimate}), stationarity and the definition of $X_n^{q,+}$ we deduce that
	\begin{equation}\label{noncentersmall}
	|p_{\e_j}^{\w}(S_j)|\leq C_{y}(\beta+\delta)
	\end{equation}
	for all $j$ large enough. By the arbitrariness of $\beta$ and $\delta$ we proved the claim.
	
	The case $x<0$ can be proved the same way using the random variables $X_n^{q,-}$ instead. It remains the case when $x=0$. For fixed $z\in\mathbb{Z}$ we consider the following sequences of random variables:
	\begin{equation*}
	Y_n^z(\w):=\sup_{k\geq n}\left|\frac{1}{2k+1}\sum_{l=-k}^{k}c_{([z]^*,l)}(\w)\right|.
	\end{equation*}
	With essentially the same arguments as above one can show that there exists a set $\Omega^z$ of full probability such that for every sequence of sides $S_j$ contained in $[z]^*\times\R$ and all $\w\in\Omega^z$ we have
	\begin{equation*}
	p_{\e_j}^{\w}(S_j)\to 0,
	\end{equation*}
	where $\Omega^z$ does not depend on the sequence $\e_j$. We finally set $\Omega^{\prime}:=\Omega^{\prime\prime}\cap\bigcap_z\Omega^z$. Let us fix $\w\in\Omega^{\prime}$. Note that if $x=0$, then for every subsequence of $\e_j$ there exists a further subsequence $\e_{j_k}$, such that either
	\begin{enumerate}
		\item[(i)] $x_{j_k}\to \pm\infty$,
		\item[(ii)] $x_{j_k}=[z]^*$ for all $k$ and for some $z\in\mathbb{Z}$
	\end{enumerate}
	In the first case we can use the construction for $x\neq 0$ with arbitrary $q\in\mathbb{Q}\cap(0,+\infty)$ since $\w\in\Omega^{\prime\prime}$ and in the second case we use that $\w\in\Omega^z$ to conclude.	
\end{proof}

\begin{rmrk}
	{\rm It is straightforward to check that the limit relation of Proposition \ref{allconvergence} holds for convergence in probability even under the weaker assumption that both the $\sigma$-algebras invariant with respect to the two group actions $\tau_{e_1},\,\tau_{\e_2}$ are trivial. 
	}
\end{rmrk}

With Proposition \ref{allconvergence} at hand we are now in a position to prove our main result.

\begin{thrm}\label{motionthm1}
	Assume that the random field $\{c_{\xi}\}_{\xi}$ is stationary and $\alpha$-mixing such that (\ref{mixingdecay}) holds. Then with probability $1$ the following holds: Let $\e_j\downarrow 0$ and let $A_{j}^0(\w)\in\mathcal{A}_{\e_j}$ be a coordinate rectangle with sides $S_{1,j}(\w),...,S_{4,j}(\w)$. Assume that $A_{j}^0(\w)$ converges in the Hausdorff metric to a coordinate rectangle $A(\w)$. Then we can choose a subsequence (not relabeled), such that $A_{\e_j}(t)(\w)$ converges locally in time to $A(t)(\w)$, where $A(t)(\w)$ is a coordinate rectangle with sides $S_i(t)(\w)$ such that $A(0)(\w)=A(\w)$ and any side $S_i(t)(\w)$ moves inward with velocity $v_i(t)(\w)$ solving the following differential inclusions:
	\begin{equation*}
	v_i(t)(\w)
	\begin{cases}
	=\frac{1}{\g}\left\lfloor\frac{2\g}{L_i(t)(\w)}\right\rfloor &\mbox{if $\frac{2\g}{L_i(t)(\w)}\notin\mathbb{N}$,}\\
	\\
	\in \frac{1}{\g}\left[\left(\frac{2\g}{L_i(t)(\w)}-1\right),\frac{2\g}{L_i(t)(\w)}\right] &\mbox{if $\frac{2\g}{L_i(t)(\w)}\in\mathbb{N}$,}
	\end{cases}
	\end{equation*}
	where $L_i(t)(\w):=\mathcal{H}^1(S_i(t)(\w))$ denotes the length of the side $S_i(t)(\w)$. The differential inclusions are valid until the extinction time when $L_i(t)(\w)=0$.
\end{thrm}

\begin{proof}
	Let $\Omega^{\prime}$ be the set of full probability given by Proposition \ref{allconvergence}. We fix $\w\in\Omega^{\prime}$. Since $A_j^0(\w)$ converges to a coordinate rectangle we can assume that the sides of $A_j^0(\w)$ are larger than $\eta>0$ for some $\eta$ independent of $j$. Therefore we can apply Proposition \ref{stillrectangle} for all $j$ large enough. For fixed $j$ and $i=1,...,4$, the minimizing movement procedure yields two random sequences $L_{i,\e_j}^k(\w),\,N_{i,\e_j}^k(\w)$. Let us denote by $L_i^{j}(t)(\w)=L_{i,\e_j}^{\lfloor \tau_j/t\rfloor}(\w)$ and $N_i^{j}(t)(\w)=N_{i,\e_j}^{\lfloor\tau_j/t\rfloor}(\w)$ the piecewise constant interpolations. Note that the function $L_i^j(t)(\w)$ is decreasing in $t$. Set
	\begin{equation*}
	t^*:=\min_i\left\{\inf\{t>0:\;\liminf_jL_i^j(t)(\w)=0\}\right\}\in[0,+\infty].
	\end{equation*}
	Recall that we already deduced from (\ref{coercivity}) that the discrete velocity, that is the distance between two corresponding sides between two time steps is equibounded by $C_{\eta}\e$ for some constant $C_{\eta}$. Thus it follows that 
	\begin{equation*}
	\min_i\liminf_{j}L_i^j(t^*)(\w)=0
	\end{equation*}
	and consequently $t^*>0$. Without changing notation we consider the subsequence realizing the $\liminf$. Then, by monotonicity, one can verify that for all $t<t^*$ we have
	\begin{equation*}
	\min_i\liminf_jL_i^j(t)(\w)>0.
	\end{equation*}
	Now fix $t_1<t^*$. Taking $i$ into account modulo 4, by construction it holds
	\begin{equation}\label{movement}
	\frac{L_{i,\e_j}^{k+1}(\w)-L_{i,\e_j(\w)}^k}{\tau}=-\frac{1}{\g}(N_{i-1,\e_j}^k(\w)+N_{i+1,\e_j}^k(\w)).
	\end{equation}
	Hence on $[0,t_1]$ the piecewise affine interpolations $t\mapsto L_i^{j,a}(t)(\w)$ are uniformly Lipschitz-continuous and decreasing while $N_i^j(t)(\w)$ is locally bounded in $L^{\infty}$. Thus, by a diagonal argument, we can find a further subsequence such that $L_i^j(t)\to L_i(t)$ pointwise and locally uniformly on $[0,t^*)$ for some locally Lipschitz-continuous, decreasing function $L_i(t)(\w)$ and additionally $N_i^j(t)(\w)$ weakly*-converges in $L^{\infty}_{\rm loc}((0,t^*))$ to some function $N_i(t)(\w)$. It follows that, up to a subsequence, $A_{\e_j}(t)(\w)$ converges in the Hausdorff metric to a coordinate rectangle $A(t)(\w)$ for all $0\leq t<t^*$.
	
	We conclude the proof by computing the velocity of each side $L_i(t)(\w)$. Again we fix $0<t<t^*$. Then $\liminf_jL_i^j(t)(\w)>0$ for all $i$. Therefore we have that the minimizers $N_{i,\e_j}^k(\w)$ of the functional $f_{\e_j}^{\w}(N)$ introduced in (\ref{minvelocity}) are uniformly bounded if $|k\tau_j-t|$ is small enough. Hence they converge, up to subsequences, to minimizers of the pointwise limit of $f^{\w}_{\e_j}$ (this can be seen as a special case of $\Gamma$-convergence on discrete spaces). By Proposition \ref{allconvergence} and the precedent discussion we know that
	\begin{equation*}
	f_{\e_j}^{\w}(N)\to\sum_{i=1}^4-2N_i+L_i(t)(\w)\mu+\frac{1}{2\g}(N_i+1)N_iL_i(t)(\w)
	\end{equation*}
	with $\mu=\mathbb{E}[c_{\xi}]$. A straightforward calculation shows that the minimizers are given by
	\begin{align*}
	N_i
	\begin{cases}
	=\left\lfloor\frac{2\g}{L_i(t)(\w)}\right\rfloor &\mbox{if $\frac{2\g}{L_i(t)(\w)}\notin\mathbb{N}$,}\vspace*{0.25cm}
	\\
	\in \left\{\frac{2\g}{L_i(t)(\w)}-1,\frac{2\g}{L_i(t)(\w)}\right\} &\mbox{otherwise.}	
	\end{cases}
	\end{align*}
	Summing the equality (\ref{movement}) we further infer that
	\begin{align}\label{integralmoton}
	L_i^{j}(t)(\w)&=L_i^j(0)(\w)-\frac{1}{\g}\sum_{k=0}^{\lfloor t/\tau_j\rfloor}\tau_j(N_{i-1}^j(k\tau_j)(\w)+N_{i+1}^j(k\tau_j)(\w))\nonumber
	\\
	&=L_i^j(0)(\w)-\frac{1}{\g}\int_0^tN_{i-1}^j(s)(\w)+N_{i+1}^j(s)(\w)\,\mathrm{d}s+{\mathcal{O}}(\tau_j).
	\end{align}
	Passing to the limit as $j\to+\infty$ in (\ref{integralmoton}), we deduce from local weak convergence that	
	\begin{equation}\label{integralform}
	L_i(t)(\w)=L_i(0)(\w)-\frac{1}{\g}\int_0^t(N_{i-1}(s)(\w)+N_{i+1}(s)(\w))\,\mathrm{d}s.
	\end{equation}	
	To conclude, we note that if $t$ is such that $2\g/L_i(t)(\w)\notin\mathbb{N}$, then by continuity we have that $2\g/L_i(t^{\prime})(\w)\notin\mathbb{N}$ for $|t-t^{\prime}|\leq\delta$ and some $\delta>0$. It follows from comparing pointwise convergence with weak*-convergence that
	\begin{equation*}
	N_i(t^{\prime})(\w)=\gamma v_i(t^{\prime})(\w)\quad\text{for almost all  } |t-t^{\prime}|\leq\delta.
	\end{equation*} 
	In particular $N_i(\cdot)(\w)$ has a constant representative on $(t-\delta,t+\delta)$ so that the velocity of the side $S_i(t)(\w)$ given by	
	\begin{equation*}
	\lim_{h\to 0}\frac{1}{2}\frac{L_{i-1}(t+h)(\w)-L_{i-1}(t)(\w)}{h}=-v_i(t)(\w)
	\end{equation*}
	exists by (\ref{integralform}) whenever $2\g/L_i(t)(\w)\notin\mathbb{N}$. Note that the formula for the velocity is true because if $2\g/L_i(t)(\w)\notin\mathbb{N}$, then at least in every short time interval opposite sides move with the same velocities. The claim for $2\g/L_i(t)(\w)\in\mathbb{N}$ follows from well known properties of weak*-convergence (note that for these values of $t$ the velocity may not be a classical derivative).	
\end{proof}

Following word by word the proof of \cite[Theorem 2]{BGN} we obtain unique limit motions in many cases:

\begin{crllr}
	Let $A_{\e}^0(\w)$ and $\{c_{\xi}\}_{\xi}$ be as in Theorem \ref{motionthm1}. Assume in addition that the lengths $L_1^0(\w),L_2^0(\w)$ of $A(\w)$ satisfy one of the three following conditions (assuming that $L^0_1(\w)\leq L^0_2(\w))$:
	\begin{enumerate}
		\item[(i)] $L_1^0(\w),L_2^0(\w)>2\g$ (total pinning),
		\item[(ii)] $L_1^0(\w)<2\g$ and $L_2^0(\w)\leq 2\g$ (vanishing in finite time with shrinking velocity larger than $1/\g$),
		\item[(iii)] $L_1^0(\w)<2\g$ and $2\g/L_1^0(\w)\notin\mathbb{N}$, and $L_2^0(\w)>2\g$ (partial pinning).
	\end{enumerate}
	Let $\e_j\to 0$. The sequence $A_{\e_j}(t)(\w)$ converges locally in time to $A(t)(\w)$, where $A(t)(\w)$ is the unique coordinate rectangle with sides $S_1(t)(\w)$ and $S_2(t)(\w)$ such that $A(0)(\w)=A(\w)$ and the side lengths $L_1(t)(\w)$ and $L_2(t)(\w)$ solve the following differential equations for all but countably many times:
	\begin{equation*}
	\begin{cases}
	\frac{d}{dt}L_1(t)(\w)=-\frac{2}{\g}\left\lfloor\frac{2\g}{L_2(t)(\w)}\right\rfloor,\vspace*{0.25cm}
	\\
	\frac{d}{dt}L_2(t)(\w)=-\frac{2}{\g}\left\lfloor\frac{2\g}{L_1(t)(\w)}\right\rfloor
	\end{cases}
	\end{equation*}
	with initial condition $L_{1}(0)(\w)=L_{1}^0(\w)$ and $L_2(0)(\w)=L_2^0(\w)$.
\end{crllr}

\begin{rmrk}\label{compactness}
	{\rm Without any assumptions on the distribution of the random field except the bound (\ref{pertubationsize}), up to subsequences we can still obtain a rectangular limit motion. Due to (\ref{only3}) we can also give an estimate of the velocity via
		\begin{equation*}
		v_i(t)(\w)\in \frac{1}{\g}\left[\left\lfloor\frac{2\g}{L_i(t)(\w)}\right\rfloor-1,\left\lfloor\frac{2\g}{L_i(t)(\w)}\right\rfloor+1\right].
		\end{equation*}
		Note that the subsequence may depend on $\w$.}
\end{rmrk}

\begin{rmrk}
	{\rm For the continuum flow it is known that rectangles always shrink to a point; see for example the more general result contained in \cite[Proposition 3.1]{Tay}. The same occurs for any possible limit motion in our discrete model provided the sets vanish in finite time. Indeed, assume by contradiction that $L_i(t^*)(\w)=0$ and $L_{i+1}(t^*)(\w)=a>0$. Then, for any $t<t^*$, by monotonicity of the side-lengths and the velocity estimate in Remark \ref{compactness} there exists a constant $c>0$ such that
		\begin{equation*}
		L_i(t^*)(\w)-L_i(t)(\w)\geq-c(t^*-t).
		\end{equation*}
		By definition of $t^*$ we obtain the bound $L_i(t)\leq c(t^*-t)$. Inserting this bound in the estimate of Remark \ref{compactness} we conclude that, again for any $0<t<t^*$ and a slightly larger constant $c>0$,
		\begin{equation*}
		L_{i+1}(t)(\w)-L_{i+1}(0)(\w)\leq -\int_{0}^{t}\frac{c}{t^*-s}\,\mathrm{d}s=c\log(1-t/t^*).
		\end{equation*}
		Letting $t\uparrow t^*$ we obtain a contradiction.
	} 
\end{rmrk}

\section{Dependence on the range of stationarity}\label{diffstatio}
In the previous section we proved that the velocity is the same as in the unperturbed deterministic case. This fact however changes if we replace the stationarity assumption on all integer shifts $\tau_z$ to a smaller subgroup since the distributions on two neighboring points in the dual lattice can be different. In particular this highlights that the results obtained hitherto are not only due to the additional scaling of the random terms but indeed due to homogenization.

\begin{dfntn}
	Let $m\in\mathbb{N}$. We say that the random field $\{c_{\xi}\}_{\xi}$ is $m$-stationary if
	\begin{equation*}
	c_{\xi}(\tau_{mz}\w)=c_{\xi+mz}(\w)\quad\quad\forall z\in\mathbb{Z}^2.
	\end{equation*}
\end{dfntn}

As we will see there are $2m$ quantities that can affect the velocity. For $k=0,...,m-1$ consider the following variables:
\begin{equation*}
c^{\text{eff}}_{k,|}(\w):=\frac{1}{m}\sum_{j=0}^{m-1}c_{([k]^*,j)}(\w),\quad\quad
\quad
c^{\text{eff}}_{k,-}(\w):=\frac{1}{m}\sum_{j=0}^{m-1}c_{(j,[k]^*)}(\w).
\end{equation*}
To obtain the velocity of the sides of a rectangle we need a generalization of Proposition \ref{allconvergence}.

\begin{prpstn}\label{convergencem}
	Assume that the random field $\{c_{\xi}\}_{\xi}$ is $m$-stationary and $\alpha$-mixing such that (\ref{mixingdecay}) holds. Then there exists a set $\Omega^{\prime}\subset\Omega$ of full probability such that for all $\w\in\Omega^{\prime}$ the following holds: Suppose that a vertical side $S_{j}$ converges in the Hausdorff sense to a limit side $S$ and that for all $j$ we have that the $x$-component $x_j$ of $S_{j}/\e_j$ fulfills
	\begin{equation}\label{goodmodulo}
	x_j=k+\frac{1}{2}\mod{m}\quad \forall j.
	\end{equation} 
	Then it holds that
	\begin{equation*}
	\lim_j p_{\e_j}^{\w}(S_j)=\mathcal{H}^1(S)\mathds{E}[c_{k,|}^{\text{eff}}].
	\end{equation*}
	Moreover the convergence is locally uniform in the following weak sense: there exists $j_0=j_0(\w)$ such that for all $j\geq j_0$ and all sequences of vertical sides $S_j^{\prime}$ such that (\ref{goodmodulo}) holds and $\text{d}_{H}(S_j,S^{\prime}_j)\leq\delta$ we have
	\begin{equation}\label{uniformspeed}
	\left|p_{\e_j}^{\w}(S^{\prime}_j)-\mathcal{H}^1(S_j^{\prime})\mathds{E}[c_{k,|}^{\text{eff}}]\right|\leq C\delta
	\end{equation}	
	for some positive constant $C>0$ independent of $S_j^{\prime}$.
	\\
	The same statement holds for horizontal sides with the condition on the $y$-component and the first moment of $c_{k,-}^{\text{eff}}$.
\end{prpstn}

\begin{proof}
	The argument to show convergence is very similar to the one used in Proposition \ref{allconvergence} restricted to a thinned dual lattice. We therefore only provide the main steps. We fix $k$ as in (\ref{goodmodulo}), set $\mu_k=\mathds{E}[c_{k,|}^{\text{eff}}]$ and define the two-sided sequence of random variables $\{z_i\}_{i\in\mathbb{Z}}$ via
	\begin{equation*}
	z_i(\w):=c_{k,|}^{\text{eff}}(\tau_{ime_2}\w).
	\end{equation*} 
	Note that this sequence is stationary and $\alpha$-mixing such that (\ref{mixingdecay}) holds. For $q\in\mathbb{Q}\cap(0,+\infty)$ we define the following average sequences:
	\begin{equation*}
	Z_n^{q,\pm}(\w):=\sup_{i\geq qn}\left|\frac{1}{2i+1}\sum_{l=-i}^iz_i(\tau_{\pm nme_1}\w)-\mu_k\right|
	\end{equation*}	
	Using $m$-stationarity and the mixing property we can argue as in the proof of Proposition \ref{allconvergence} to show that there exists a set $\Omega^{\prime\prime}$ of full probability such that all the sequences $Z_n^{q,\pm}$ converge to $0$ pointwise on $\Omega^{\prime\prime}$. Up to minor changes the proof of convergence now is the same as for Proposition \ref{allconvergence}. We omit the details.
	
	In order to prove (\ref{uniformspeed}) we have to distinguish two cases: First assume that the $x$-coordinate (also denoted by $x$) of $S$ is positive (the case of negative $x$-coordinate works the same way). Then, for $\delta$ small enough (otherwise (\ref{uniformspeed}) is trivial), we have $x_j^{\prime}>0$ for $j$ large enough depending only on $S_j$. The key is to show that we can compare $S_j^{\prime}$ to one of the sequences of random variables $Z_n^{q,+}$ as in the proof of Proposition \ref{allconvergence}, where $q$ can be chosen only depending on the sequence $S_j$. Then the speed of convergence is determined by the one of $Z_n^{q,+}$ for one particular $q$.
	
	We start with the case of a vertically centered side $S$, that means $S=\{x\}\times[-\mathcal{H}^1(S)/2,\mathcal{H}^1(S)/2]$. Given $0<\beta<<\delta$ there exists $j_0$ such that for all $j\geq j_0$ we have $x+\beta\geq \e_jx_j$ and $\e_j\#\{\xi\in S_j/\e_j\cap\mathcal{Z}^2\}\geq\mathcal{H}^1(S)-\beta$. Using the assumption $\text{d}_{H}(S_j,S^{\prime}_j)\leq\delta$, a straightforward computation yields
	\begin{equation*}
	\#\{\xi\in \frac{S^{\prime}_j}{\e_j}\cap\mathcal{Z}^2\}\geq \frac{\mathcal{H}^1(S)-2\delta-\beta}{x+\delta+\beta}x_j^{\prime}
	\end{equation*} 
	for all $j\geq j_0$. Therefore we have to chose $\frac{\mathcal{H}^1(S)-2\delta-\beta}{x+\delta+\beta}>3q$ which can be done uniformly for small $\delta$. Moreover, from our assumptions we deduce
	\begin{equation*}
	\#\left(\left\{\xi\in \frac{S^{\prime}_j}{\e_j}\cap\mathcal{Z}^2\right\}\Delta\left\{\xi=(x_j,l):\;|l|\leq\frac{\mathcal{H}^1(S^{\prime}_j)}{2\e_j}\right\}\right)\leq \frac{4\delta}{\e_j}.
	\end{equation*}
	Assuming (\ref{boundedsetting}) we deduce that $\sup_j\mathcal{H}^1(S_j^{\prime})\leq C$. Hence we can argue as in (\ref{withq}) to prove that
	\begin{equation}\label{qcontrol}
	\left| p^{\w}_{\e_j}(S_j^{\prime})-\mathcal{H}^1(S_j^{\prime})\mu_k\right|\leq C\delta+Z_{n_j}^{q,+}(\w),
	\end{equation}	
	where $[n_j]^*=x^{\prime}_j$. Since $|x_j-x_j^{\prime}|\leq \delta/\e_j$ and $\e_jx_j\to x$, for every $n\in\mathbb{N}$ we can find $j_0$ (depending only on $S_j$) such that for all $j\geq j_0$ we have
	\begin{equation*}
	x_j^{\prime}\geq \frac{x/2-\delta}{\e_j}\geq n.
	\end{equation*} 
	Hence $n_j\to +\infty$ and since $Z_n^{q,+}$ converges to $0$ on $\Omega^{\prime\prime}$, (\ref{uniformspeed}) holds in this particular case.
	
	The case of a general side $S=\{x\}\times[y-\mathcal{H}^1(S)/2,y+\mathcal{H}^1(S)/2]$ with $x,y>0$ can be treated with the same arguments as in the derivation of (\ref{noncentersmall}) since this construction is uniformly with respect to small displacements of the limit side. We leave out the details here.
	
	We are left with the case when $x=0$. Again it is enough to consider a centered side $S$ since the other cases can be deduced from this one. Let us take $q$ small enough such that 
	\begin{equation*}
	\frac{\mathcal{H}^1(S)-4\delta}{2\delta}>3q.
	\end{equation*}
	By construction there exists $j_0$ such that for every sequence $S^{\prime}_j$ fulfilling the assumptions we have $\mathcal{H}^1(S^{\prime}_j)/2\e_j>q |x_j^{\prime}|$ for all $j\geq j_0$. Thus, if $x_j^{\prime}$ is not bounded we can control the speed of convergence with the random variables $Z_n^{q,\pm}$ as in (\ref{qcontrol}). Perhaps after enlarging $j_0$, we obtain that
	\begin{equation*}
	|Z_{j}^{q,\pm}(\w)|\leq\delta\quad\forall j\geq j_0.
	\end{equation*}	
	The estimate (\ref{uniformspeed}) now follows from distinguishing the case where $|x_j^{\prime}|>j_0$ for which we can use the above bound and (\ref{qcontrol}) or $|x_j^{\prime}|\leq j_0$ where we have to control finitely many sequences of random variables that converge to $0$ as $S_j\to S$. 	
\end{proof}

Before we state our next theorem, let us derive a suitable expression for the velocity. We remark that due to Proposition \ref{convergencem} the argument is similar to the deterministic case treated in \cite{BrSc}. To reduce notation, we set $\mu_k=\mathds{E}[c_{k,|}^{\text{eff}}]$ and $\lambda_k=\mathds{E}[c_{k,-}^{\text{eff}}]$ and identify the indices modulo $m$ whenever necessary.

We have to minimize the function $v_{i,\e}^{\w}(N)$ given by (\ref{sidevelocity}) which is the correct one describing the velocity if the limit function as $\e_j\to 0$ has a unique minimizer. For the moment we restrict the analysis to the left vertical side. Up to a subsequence, we have that the $x$-component of $s_{i,\e_j}/\e_j$ is constant modulo $m$, that is there exists $n\in 0,...,m-1$ such that
\begin{equation*}
x^i_j=n+\frac{1}{2}\mod m\quad\forall j.
\end{equation*}
If $s_{i,\e_j}$ converges to a limit side of length $L$, then by Proposition \ref{convergencem} we have that along this particular subsequence, it holds that
\begin{align}\label{limitvelocitym}
v_{i,\e_j}^{\w}(N)\to v_i^{n,L}(N):=-2N+L\mu_k+\frac{L}{2\g}(N+1)N &&\mbox{if $N+n=k\mod m$.}
\end{align}
As we will show in the following, we can define an effective velocity which does not depend on the particular subsequence. Setting $N^*=\lfloor 2\g/L\rfloor$, as an analogue of (\ref{only3}) we have
\begin{equation}\label{just3}
\min_N v_i^{n,L}(N)=\min\{v_i^{n,L}(N^*),\,v_i^{n,L}(N^*+1),\,v_i^{n,L}(N^*-1)\}.
\end{equation} 
Since a precise analysis of the minimization process is only possible provided the limit functional has a unique minimizer, let us check when this is the case. There are three equivalences that turn out to be useful to characterize the lack of uniqueness. Write $N^*=2\g/L-\xi$ with $\xi\in[0,1)$ and suppose that $N^*+n=k^*\mod m$. Then it holds
\begin{align}
& v_i^{n,L}(N^*)\leq v_i^{n,L}(N^*+1)&\iff& \quad\quad\xi\leq 1+\g(\mu_{k^*+1}-\mu_{k^*}),&\nonumber
\\
& v_i^{n,L}(N^*)\leq v_i^{n,L}(N^*-1)&\iff& \quad\quad\xi\geq\g(\mu_{k^*}-\mu_{k^*-1}),&\label{charmin}
\\
& v_i^{n,L}(N^*+1)\leq v_i^{n,L}(N^*-1)&\iff& \quad\quad\xi\geq\frac{1}{2}+\frac{\g}{2}(\mu_{k^*+1}-\mu_{k^*-1}).&\nonumber
\end{align}
Thus minimizers are not unique if and only if
\begin{itemize}
	\item[(i)] $\g(\mu_{k^*+1}-\mu_{k^*})+1=\xi\geq\frac{\g}{2}(\mu_{k^*+1}-\mu_{k^*-1})+\frac{1}{2}$,
	\item[(ii)] $\g(\mu_{k^*}-\mu_{k^*-1})=\xi\leq \frac{\g}{2}(\mu_{k^*+1}-\mu_{k^*-1})+\frac{1}{2}$,
	\item[(iii)] $\xi=\frac{1}{2}=\g(\mu_{k^*}-\mu_{k^*-1})=\g(\mu_{k^*}-\mu_{k^*+1})$, 
\end{itemize}
where we left out those inequalities with no information. Due to (\ref{pertubationsize}) the third possibility cannot occur and also the inequalities in (i) and (ii) are always fulfilled since $\mu_{k+1}-\mu_{k-1}=(\mu_{k+1}-\mu_k)+(\mu_k-\mu_{k-1})$. In particular the set of side lengths where the minimization problem (\ref{just3}) has not a unique solution is discrete. The same analysis for the remaining sides yields the following singular side lengths:
\begin{align*}
\mathcal{S}^l_{|}&:=\left\{L\in (0,+\infty):\;2\g/L\in \mathbb{N}_0+\g(\mu_k-\mu_{k-1})\quad\text{for some }k\right\},
\\
\mathcal{S}^r_{|}&:=\left\{L\in (0,+\infty):\;2\g/L\in \mathbb{N}_0-\g(\mu_k-\mu_{k-1})\quad\text{for some }k\right\},
\\
\mathcal{S}^l_{-}&:=\left\{L\in (0,+\infty):\;2\g/L\in \mathbb{N}_0+\g(\lambda_k-\lambda_{k-1})\quad\text{for some }k\right\},
\\
\mathcal{S}^u_{-}&:=\left\{L\in (0,+\infty):\;2\g/L\in \mathbb{N}_0-\g(\lambda_k-\lambda_{k-1})\quad\text{for some }k\right\}.
\end{align*}
Whenever it is clear from the context, we associate to a side $S_i$ the corresponding set $\mathcal{S}_i\in\{\mathcal{S}^l_{|},\mathcal{S}^r_{|},\mathcal{S}^l_{-},\mathcal{S}^u_{-}\}$.

Now let us analyze the minimization scheme. Again we illustrate the procedure only for the left vertical side. To this end we fix $L\notin\mathcal{S}^l_{|}$. Setting $X_0=x_j^i$, we will see that the motion of the corresponding left vertical side will be given locally by the following algorithm:
\\
For $l=0,1,...$ set
\begin{enumerate}
	\item[] $n_l:=X_l-\frac{1}{2}\mod m,$
	\item[] $N_{l+1}=\underset{N}{\text{argmin }}v_i^{n_l,L}(N),$
	\item[] $X_{l+1}:=X_l+N_{l+1},$
\end{enumerate}
where $v_i^{n,L}$ is defined in (\ref{limitvelocitym}). This algorithm is well-defined for $L\notin\mathcal{S}_{|}^l$ and gives rise to an effective velocity as shown in the lemma below:
\begin{lmm}\label{homogenvelocity}
	There exist nonnegative integer numbers $\tilde{l},\,T,\,M$ such that $\tilde{l}+T\leq m$ and 
	\begin{equation*}
	X_{l+T}-X_l=Mm\quad\forall l\geq\tilde{l}.
	\end{equation*}	
	Moreover, the quotient $M/T$ does not depend on $X_0$.
\end{lmm}

\begin{proof}
	Observe that the quotient space $\mathbb{Z}_{/m\mathbb{Z}}$ has only $m$ distinct elements so that the first statement holds. For the second statement we first establish a monotonicity property of the orbits with respect to the initial data $X_0$. To this end let $X_0\leq X_0^{\prime}$. We argue inductively. Due to (\ref{just3}) it is clear that $X_1\leq X_1^{\prime}$ in case that $X_0=X_0^{\prime}$ or $X_0^{\prime}-X_0\geq 2$. It remains the case where $X_0^{\prime}-X_0=1$. We assume by contradiction that $X_1>X_1^{\prime}$. Writing $N^*=2\g/L-\xi$, the minimizer to determine $X_1$ would be given by $N^*+1$ while for $X_1^{\prime}$ minimizing yields $N^*-1$. Using minimality one easily derives that in this case we have
	\begin{equation*}
	\g(\mu_{k^*+1}-\mu_{k^*})\geq\xi\geq \g(\mu_{k^*+1}-\mu_{k^*})+1,
	\end{equation*}
	where $k^*=X_0-\frac{1}{2}+N^*\mod m$. This gives a contradiction. By iteration we obtain that $X_k\leq X_k^{\prime}$ for all $k$. Now we argue as in Proposition 3.6 in \cite{BrSc} by comparing the long-time behavior of the orbits with starting points $X_0,X_0^{\prime}$ and $X_0+m$. For $L,l_0\in\mathbb{N}$ we let $k=l_0+LT(x_0)T(x_0^{\prime})$. By the first part of the proof and orbit monotonicity, for $l_0$ large enough it holds that
	\begin{equation*}
	X_{l_0}+LT(x_0^{\prime})M(x_0)m\leq X^{\prime}_{l_0}+LT(x_0)M(x_0^{\prime})m\leq X_{l_0}+LT(x_0^{\prime})M(x_0)m+m.
	\end{equation*}
	Dividing this inequality by $L$ and letting $L\to +\infty$ yields the claim.
\end{proof}

\begin{dfntn}
	For a given type of side with length $L\notin \mathcal{S}_i$, let $M_i,T_i$ be as in Lemma \ref{homogenvelocity}, where $T_i$ is chosen to be minimal. The effective velocity for a side $S_i$ is defined as a function $v_{i}^{\text{eff}}:(0,+\infty)\backslash\mathcal{S}_i\rightarrow [0,+\infty)$ by
	\begin{equation*}
	v_i^{\text{eff}}(L)=\frac{M_im}{T_i}.
	\end{equation*}	
	In view of Lemma \ref{homogenvelocity}, this function is well-defined.  
\end{dfntn}

\begin{rmrk}
	{\rm In contrast to the deterministic environments considered in \cite{BrSc,Scilla} in our setting the effective velocity of two opposite sides can be different. However this is not due to random effects but can already be caused by a slightly more complex periodic structure as shown in the following example.}
\end{rmrk}

\begin{xmpl}\label{centermoves}
	Let $m=6$ and let $c_{\xi}$ be a (maybe deterministic) field such that 
	\begin{align*}
	\mu_0=-\frac{1}{8\g},\quad\mu_1=\mu_2=\mu_5=\frac{1}{8\gamma}\quad\mu_3=\mu_4=0.
	\end{align*}
	If $2\g/L\in (3-\frac{1}{8},3)$, then the left side of a rectangle moves faster than the right side, namely 
	\begin{equation*}
	v_i^{\text{eff}}(L)=3>2=v_{i+2}^{\text{eff}}(L).
	\end{equation*}
\end{xmpl}

\begin{proof}
	This follows from a straightforward computation based on the minimality criteria (\ref{charmin}). Indeed, if the left side starts at $n_0=0$, then we have $N_1=N_2=3$. If the right side starts also at $n_0=0$ we deduce that $N_1=N_2=N_3=2$. We leave the details of the computation to the interested reader.	
\end{proof}

Let us now compute the pinning threshold, that is the critical side length above which a side does not move after a finite number of time steps (or equivalently $v_i^{\text{eff}}(L)=0$). Due to (\ref{just3}) a necessary condition is given by $L>\g$.  We then have to compare the values of $N\in\{0,1,2\}$. For an arbitrary starting position of a left vertical side we obtain the conditions
\begin{equation*}
L>\frac{2\g}{1+\g(\mu_{k+1}-\mu_k)},
\quad\quad
L>\frac{4\g}{3+\g(\mu_{k+2}-\mu_k)}.
\end{equation*}   
As we can chose the index $k$, the pinning threshold for a left vertical side is given by
\begin{equation*}
\overline{L}_i=\min_k\left\{\max\left\{\frac{2\g}{1+\g(\mu_{k+1}-\mu_k)},\frac{4\g}{3+\g(\mu_{k+2}-\mu_k)}\right\}\right\}>\g.
\end{equation*}
The pinning thresholds for the other sides are given by
\begin{align*}
\overline{L}_{i+1}&=\min_k\left\{\max\left\{\frac{2\g}{1+\g(\lambda_{k-1}-\lambda_{k})},\frac{4\g}{3+\g(\lambda_{k-2}-\lambda_{k})}\right\}\right\},
\\
\overline{L}_{i+2}&=\min_k\left\{\max\left\{\frac{2\g}{1+\g(\mu_{k-1}-\mu_k)},\frac{4\g}{3+\g(\mu_{k-2}-\mu_k)}\right\}\right\},
\\
\overline{L}_{i+3}&=\min_k\left\{\max\left\{\frac{2\g}{1+\g(\lambda_{k+1}-\lambda_k)},\frac{4\g}{3+\g(\lambda_{k+2}-\lambda_k)}\right\}\right\},
\end{align*}
where the indices rotate clockwise. The next lemma contains some properties of the effective velocities. We remark that the same results have been obtain in \cite{BrSc} but we find it difficult to reproduce the argument in our slightly more complex setting. Therefore we provide a different proof.
\begin{lmm}\label{properties}
	The velocity functions $v_{i}^{\text{eff}}$ satisfy the following properties:	
	\begin{enumerate}
		\item[(a)] $v_{i}^{\text{eff}}$ is constant on each interval contained in $(0,+\infty)\backslash\mathcal{S}_i$.
		\item[(b)] $v_i^{\text{eff}}(L)=0$ if $L>\overline{L}_i$.
		\item[(c)] $v_i^{\text{eff}}(\cdot)$ is non-increasing in $L$.
	\end{enumerate}
\end{lmm}
\begin{proof}
	To prove the first assertion, fix an interval $I\subset (0,+\infty)\backslash\mathcal{S}_i$ and let $L\in I$. We claim that there exists an open interval $I_L$ around $L$ such that for all $n$ and all $L^{\prime}\in I_L$ the unique minimizers of $v_i^{n,L^{\prime}}$ agree with the unique minimizer of $v_i^{n,L}$. As $I$ is connected, it then follows that the minimizers are the same for all $L^{\prime}\in I$ and we conclude by iteration. To prove the claim, it is enough to observe that whenever $L_j\to L$, it follows that $v_i^{n,L_j}(N)\to v_i^{n,L}(N)$ pointwise. Due to (\ref{just3}) also the minimizers are bounded. By uniqueness they converge to the minimizer of the limit function.
	
	The second assertion is an immediate consequence of the definition of the pinning threshold.
	
	To prove the monotonicity, take $L>L^{\prime}$. The claim follows from the fact that, for every $n$, in a multi-valued sense it holds that
	\begin{equation}\label{velmonotonicity}
	\underset{N}{\text{argmin }}v_i^{n,L}(N)\leq \underset{N}{\text{argmin }}v_i^{n,L^{\prime}}(N).
	\end{equation}
	Indeed, observe that $N(L):=\lfloor 2\gamma/L\rfloor\leq\lfloor 2\gamma/L^{\prime}\rfloor=:N(L^{\prime})$. Then by (\ref{just3}) it suffices to treat the two cases $N(L)=N(L^{\prime})$ and $N(L)+1=N(L^{\prime})$. In any case, again applying (\ref{just3}) there are only finitely many options for violating (\ref{velmonotonicity}) that can be ruled out by a direct calculation based on a characterization as in (\ref{charmin}). We omit the details.
\end{proof}
Now we are in a position to state the main theorem for $m$-stationary fields under the same $\alpha$-mixing hypothesis as in Theorem \ref{motionthm1}.

\begin{thrm}\label{mainmstatio}
	Assume that the random field $\{c_{\xi}\}_{\xi}$ satisfies (\ref{pertubationsize}), is $m$-stationary and $\alpha$-mixing such that (\ref{mixingdecay}) holds. Then with probability $1$ the following holds: Let $\e_j\downarrow 0$ and let $A_{j}^0(\w)\in\mathcal{A}_{\e_j}$ be a coordinate rectangle with sides $S_{1,j}(\w),...,S_{4,j}(\w)$. Assume that $A_{j}^0(\w)$ converges in the Hausdorff metric to a coordinate rectangle $A(\w)$. Then we can choose a subsequence (not relabeled), such that $A_{\e_j}(t)(\w)$ converges locally in time to $A(t)(\w)$, where $A(t)(\w)$ is a coordinate rectangle with sides $S_i(t)(\w)$ such that $A(0)(\w)=A(\w)$. Each side $S_i(t)(\w)$ moves inward with velocity $v_i(t)(\w)$ solving the following inclusions:
	\begin{equation*}
	v_i(t)(\w)
	\begin{cases}
	=\frac{1}{\g}v_{i}^{\text{eff}}(L_i(t)(\w)) &\mbox{if $L_i(t)(\w)\notin\mathcal{S}_{i}$,}\\\\
	\in \frac{1}{\g}\left[(v_{i}^{\text{eff}})^{(-)}(L_i(t)(\w)),(v_{i}^{\text{eff}})^{(+)}(L_i(t)(\w))\right] &\mbox{otherwise,}
	\end{cases}
	\end{equation*}
	where $L_i(t)(\w):=\mathcal{H}^1(S_i(t)(\w))$ denotes the length of the side $S_i(t)(\w)$. The inclusions are valid until the extinction time when $L_i(t)(\w)=0$.
\end{thrm}

\begin{proof}
	Due to Remark \ref{compactness} we only have to derive the formula for the velocities. We fix $\w\in\Omega^{\prime}$ given by Proposition \ref{convergencem}. Using the same notation as in the proof of Theorem \ref{motionthm1}, we have to identify the weak*-limit $N_i$ of $N_i^j(\cdot)$ on the interval $(0,t^*)$. Therefore we fix $t_1\in (0,t^*)$ such that $L_i(t_1)(\w)\notin\mathcal{S}_i$. Given $\delta>0$ there exists an open interval $I_{\delta}\ni t_1$ and $j_0$ such that for all $j\geq j_0$
	\begin{enumerate}
		\item[(i)] $L_{i}^j(t)(\w)\notin \mathcal{S}_i\quad\forall t\in I_{\delta}$,
		\item[(ii)] $\text{d}_H(S_{i,j}(t)(\w),S_{i,j}(t_1)(\w))\leq\delta\quad\forall t\in I_{\delta}$.
	\end{enumerate}	
	Hence, by Proposition \ref{convergencem} we may assume that for $j\geq j_0$ and $t\in I_{\delta}$ there exists $n=n(j,t)$ such that for $L=L_{i}^j(t)(\w)$
	\begin{equation*}
	N_i^j(t)(\w)=\underset{N}{\text{argmin }}v_i^{n,L}(N),
	\end{equation*}
	where $v_i^{n,L}$ is defined in (\ref{limitvelocitym}). Since without loss of generality $L_i^j(t)(\w)$ is in the same interval contained in $(0,+\infty)\backslash\mathcal{S}_i$ as $L_i^j(t_1)(\w)$, we infer from the Lemmata \ref{homogenvelocity} and \ref{properties} (a) that
	\begin{align}
	\int_{I_{\delta}}N_i(s)(\w)\,\mathrm{d}s&=
	\lim_j\int_{I_{\delta}}N_i^j(s)(\w)\,\mathrm{d}s=\lim_j\sum_{k\tau_j\in I_{\delta}}\tau_jN_i^j(k\tau_j)(\w)+{\mathcal{O}}(\tau_j)\nonumber
	\\
	&=\lim_j |I_{\delta}|v_i^{\text{eff}}(L_i(t_1)(\w))+{\mathcal{O}}(\tau_j)=|I_{\delta}|v_i^{\text{eff}}(L_i(t_1)).\label{passlimit}
	\end{align}
	Dividing by $|I_{\delta}|$ and letting $\delta\to 0$ we obtain the claim using Lebesgue's differentiation theorem. Note that similar to the proof of Theorem \ref{motionthm1} the formula for the velocity holds for every such $t_1$ since $N_i$ has a constant representative locally near $t_1$ so that the side positions are differentiable in the classical sense. However here we have to take the side positions and cannot deduce the velocity from the side lengths since the center might move (see Example \ref{centermoves}).
	
	It remains the case where $L_i(t_1)(\w)\in\mathcal{S}_i$. Note that by (\ref{velmonotonicity}) we still have the monotonicity of orbits. That means if $L^-,L^+\in(0,+\infty)\backslash\mathcal{S}_i$ are in the two intervals enclosing $L_i(t_1)(\w)$ such that $L^-<L_i(t_1)(\w)<L^+$ and we start the algorithm for computing the effective velocity with the same initial datum choosing the minimizer arbitrarily in the case of non-uniqueness, we have
	\begin{equation*}
	X^+_k\leq X_k\leq X^-_k.
	\end{equation*}
	This yields	
	\begin{equation*}
	|I_{\delta}|v_i^{\text{eff}}(L^+)\leq\int_{I_{\delta}}N_i(s)(\w)\,\mathrm{d}s\leq|I_{\delta}|v_i^{\text{eff}}(L^-).
	\end{equation*}	
	The claim follows after dividing by $|I_{\delta}|$, sending $\delta\to 0$ and then taking both the limits as $L^-\uparrow L_i(t_1)(\w)$ and $L^+\downarrow L_i(t_1)(\w)$ for which we use monotonicity.	
\end{proof}

Again we have several cases where a unique limit motion exists. However the equations differ since the velocity of two opposite sides may be not equal. We don't list all possible cases where there is a unique motion.

\begin{crllr}
	Let $A_{\e}^0(\w)$ and $\{c_{\xi}\}_{\xi}$ be as in Theorem \ref{mainmstatio}. Assume in addition that the lengths $L_1^0(\w),L_2^0(\w)$ of $A(\w)$ satisfy one of the three following conditions (we assume that $L^0_1(\w)\leq L^0_2(\w)$ and $\overline{L}_1\leq\overline{L}_3$ as well as $\overline{L}_2\leq\overline{L}_4$):
	\begin{enumerate}
		\item[(i)] $L_i^0(\w)>\overline{L}_i$ (total pinning),
		\item[(ii)] $L_1^0(\w)<\overline{L}_1$ and $L_2^0(\w)\leq \overline{L}_2$ (vanishing in finite time),
		\item[(iii)] $\overline{L}_1<L_1^0(\w)<\overline{L}_3$ and $L_1^0(\w)\notin\mathcal{S}_3$, and $L_2^0(\w)>\overline{L}_4$ (partial pinning).
	\end{enumerate}	
	Then with probability $1$ the following holds: Let $\e_j\to 0$. The sequence $A_{\e_j}(t)(\w)$ converges locally in time to $A(t)(\w)$, where $A(t)(\w)$ is the unique coordinate rectangle with sides $S_i(t)(\w)$ such that $A(0)(\w)=A(\w)$ and the side lengths $L_i(t)(\w)$ solve the following differential equations for all but countably many times:
	\begin{equation*}
	\frac{d}{dt}L_i(t)(\w)=-\frac{1}{\g}\left(v_{i-1}^{\text{eff}}(L_{i-1}(t)(\w))+v_{i+1}^{\text{eff}}(L_{i+1}(t)(\w)\right)
	\end{equation*}	
	with initial condition $L_{1}(0)(\w)=L_{1}^0(\w)$ and $L_2(0)=L_2^0$.
\end{crllr}

\begin{proof}
	(i) and (ii) can be proven as in Theorem 3.2 in \cite{BGN}. In Case (iii) note that the side $S_3$ moves inward with a strictly positive velocity bounded away from $0$. Hence $L_2(t)(\w)$ is strictly decreasing until it vanishes. Consequently $L_2(t)(\w)\in\mathcal{S}_2\cup\mathcal{S}_4$ only for countably many times. Moreover, as soon as $L_2(t)(\w)<\overline{L}_4$ also the side length $L_1(t)(\w)$ shrinks strictly since from that time on the side $S_4$ moves inward with positive velocity. Hence the times when $L_1(t)(\w)\in\mathcal{S}_1\cup\mathcal{S}_3$ are discrete, too. Note that by continuity, the values at the critical times are uniquely defined. In between these critical times, one can use general results from ODE-theory to obtain that the rectangular motion is unique. The particular form of the ODE describing the motion is a straightforward consequence of Theorem \ref{mainmstatio}.	
\end{proof}

\subsection{An outlook for possible homogenization in time}\label{lastexample}
In this last section we show that, under certain assumptions, the random field considered in Example \ref{nonconv} exhibits an averaged velocity as well. We don't aim at giving results in full detail since Example \ref{nonconv} only serves as a toy model and the case of only stationary perturbations seems much more involved and we are not sure if homogenization can be proved. 

To be precise we give a hint how to generalize Lemma \ref{homogenvelocity} in a probabilistic way. Note that by construction of the perturbations in Example \ref{nonconv} there is nothing left to prove for horizontal sides since we may apply Proposition \ref{allconvergence}. For vertical sides we first indicate some possible uniqueness issues. For the moment let us neglect the error terms in (\ref{errors}). Denoting by $x_j=x_j(\w)$ the $x$-component of $s_{i,\e_j}(\w)/\e_j$, for left vertical sides we have to minimize
\begin{equation*}
v_{i,\e_j}^{\w}(N)=-2N+\frac{l_{i,\e_j}}{2\g}(N+1)N+l_{i,\e_j}X_{\lfloor x_j+N\rfloor}(\w).
\end{equation*}
Let us take a closer look at the non-uniqueness of minimizers. Again we set $N_j^*=\lfloor 2\g/l_{i,\e_j}\rfloor$ and $k_j:=\lfloor x_j+N^*_j\rfloor$. Writing $N^*_j=2\g/l_{i,\e_j}-\xi_j$, as in (\ref{charmin}) we deduce that minimizers are not unique if
\begin{equation*}
\xi_j\in\left\{\g(X_{k_j}(\w)-X_{k_j-1}(\w)),1+\g(X_{k_j+1}(\w)-X_{k_j}(\w))\right\}.
\end{equation*}

Without any further assumptions, as $j$ varies this set can be dense in a whole interval as the following example shows.

\begin{xmpl}
	Let $X_k$ be uniformly distributed on the interval $(0,1)$. Then, by independence, for every $k$ the random variable $Y_k=X_k-X_{k-1}$ has a triangular distribution on $(-1,1)$. Hence the sequence $\{Y_{2k}\}_k$ is an independent and identically distributed sequence of random variables. Then
	\begin{equation*}
	\mathbb{P}\left(\w:\;\{Y_{2k}(\w)\}_k\text{ is not dense in }(-1,1)\right)=0.
	\end{equation*}
	Indeed, given $q\in\mathbb{Q}\cap(-1,1)$ and $n\in\mathbb{N}$, from independence we infer
	\begin{equation*}
	\mathbb{P}\left(\w:\;Y_{2k}(\w)\notin q+(-\frac{1}{n},\frac{1}{n})\quad\forall k\right)=0.
	\end{equation*}
\end{xmpl}
This example indicates that a precise analysis for the limit velocity is quite difficult (of course one has to take into account also the error terms (\ref{errors})). Instead, if we assume that $X_k$ takes only finitely many values, then the set of side lengths $\mathcal{S}$ where the minimization problem has not a unique solution is again discrete. We now give a formal argument how one can treat this case. For side lengths with unique corresponding minimizers, we can indeed neglect (\ref{errors}). Moreover, by the same topological argument used in the proof of Lemma \ref{properties}, the minimization does not depend on the particular side length in one interval contained in $(0,+\infty)\backslash\mathcal{S}$. However, in contrast to the mixing case, the choice of minimizers is still random. Given $L\notin\mathcal{S}$ we have to consider the following (now random) algorithm: Given a starting point $P_0=x_j$,
\\
for $l=0,1,...$ set
\begin{enumerate}
	\item[] $n_l(\w):=P_l(\w)-\frac{1}{2},$
	\item[] $N_{l+1}(\w)=\underset{N}{\text{argmin }}\left\{-2N+\frac{L}{2\g}(N+1)N+LX_{n_l(\w)+N}(\w)\right\},$
	\item[] $P_{l+1}(\w):=P_l(\w)+N_{l+1}(\w).$
\end{enumerate}

Note that if $L$ is below the corresponding pinning threshold (which can easily be estimated since the random variables take only finitely many values), then the sequence $\{N_l\}_l$ is identically distributed and has a finite range dependence. Therefore, by the strong law of large numbers, almost surely we have

\begin{equation*}
\lim_{l_1\to+\infty}\frac{1}{l_1}\sum_{l=1}^{l_1}N_l(\w)=\mathds{E}[N_0].
\end{equation*}

Note that the limit does not depend on $x_j$. Moreover, as a trivial remark we can make the exceptional set independent of the starting position $x_j$. In order to prove that $\mathds{E}[N_0]$ is, up to a multiplicative constant, the velocity of the left vertical side, we need to control the speed of convergence independently of $x_j$. This can be achieved by defining finitely many stochastic processes (for a fixed limit side) similar to the proof of Proposition \ref{convergencem} using the fast decay of error probabilities due to finite range dependence. Since we only want to give a possible outlook we don't go into details here. Finally one can argue as in the proof of Theorem \ref{mainmstatio} and pass to the limit in the integral in (\ref{passlimit}). We leave the computation to the interested reader.

\subsection*{Acknowledgement}
The author thanks Marco Cicalese for suggesting to look at geometric minimizing movements in random media and for useful advices while writing this paper.	


\bibliographystyle{plain}

\end{document}